\documentclass[10pt]{article}
\usepackage{hyperref}
\usepackage{amsmath}
\usepackage[dvips]{graphicx}
\usepackage{amsthm}
\usepackage{epsfig}
\usepackage{amssymb}
\usepackage{dsfont}

\newtheorem{theorem}{Theorem}[section]

\newtheorem{corollary}[theorem]{Corollary}
\newtheorem{proposition}[theorem]{Proposition}

\catcode`@=11
\@addtoreset{equation}{section}
\catcode`@=12

\begin{document}

\title{Sharp criteria of Liouville type for some nonlinear systems}
\author{Yutian Lei and Congming Li}

\date{}
\maketitle

\begin{abstract}
In this paper, we establish the sharp criteria for the nonexistence of positive solutions to the Hardy-Littlewood-Sobolev (HLS) type system of nonlinear equations and the corresponding nonlinear differential systems of Lane-Emden type equations. These nonexistence results, known as Liouville type theorems, are fundamental in PDE theory and applications. A special iteration scheme, a new shooting method and some Pohozaev type identities in integral form as well as in differential form are created. Combining these new techniques with some observations and some critical asymptotic analysis, we establish the sharp criteria of Liouville type for our systems of nonlinear equations. Similar results are also derived for the system of Wolff type integral equations and the system of $\gamma$-Laplace equations. A dichotomy description in terms of existence and nonexistence for solutions with finite energy is also obtained.
\end{abstract}



\section{Introduction }

In this paper, we establish sharp criteria for existence
and nonexistence of positive solutions to the Hardy-Littlewood-Sobolev
(HLS) system of nonlinear equations
\begin{equation}
 \left \{
   \begin{array}{l}
   u(x)= \displaystyle\int_{R^n}\frac{v^q(y)dy}{|x-y|^{n-\alpha}}\\
   v(x)=\displaystyle\int_{R^n}\frac{u^p(y)dy}{|x-y|^{n-\alpha}},
   \end{array}
   \right.\label{1.3s}          
 \end{equation}
and the corresponding nonlinear differential systems of Lane-Emden type equations
\begin{equation} \label{1.1s}
 \left \{
   \begin{array}{l}
      (-\Delta)^k u=v^q, ~u,v>0,\\
      (-\Delta)^k v=u^p, ~p,q>0.
   \end{array}
   \right.
\end{equation}
These systems are the `blow up' equations for a large class of systems of
nonlinear equations arising from geometric analysis, fluid dynamics,
and other physical sciences. The nonexistence of positive solutions for systems of
`blow up' type like (\ref{1.3s}) and (\ref{1.1s}), known as Liouville type theorem,
is useful in deriving existence, a priori estimate, regularity and asymptotic
analysis of solutions. Another important topic is the study of the Wolff
type system of nonlinear equations:
\begin{equation}
 \left \{
   \begin{array}{l}
      u(x)=W_{\beta,\gamma}(v^q)(x)\\
      v(x)=W_{\beta,\gamma}(u^p)(x),
   \end{array}
   \right. \label{1.4s}         
 \end{equation}
and the corresponding system of $\gamma$-Laplace equations
\begin{equation}
 \left \{
   \begin{array}{l}
      -\Delta_\gamma u(x)=v^q(x), \quad x \in R^n,\\
      -\Delta_\gamma v(x)=u^p(x), \quad x \in R^n.
   \end{array}
   \right.\label{1.5s}          
 \end{equation}

Recall a Liouville-type theorem for the Lane-Emden equation
\begin{equation} \label{1.1}
-\Delta u=u^p, \quad in ~ R^n ~(n \geq 3)
\end{equation}
obtained by Caffarelli, Gidas and Spruck \cite{CGS}: if $p \in
(0,\frac{n+2}{n-2})$, then (\ref{1.1}) has no positive classical
solution. When $p \geq \frac{n+2}{n-2}$, (\ref{1.1}) has positive
classical solution. Namely, the right end point $\frac{n+2}{n-2}$
is a sharp criterion distinguishing the existence and the nonexistence.
Numbers like this, separating the existence and the nonexistence, are
called the critical exponents.

In the following theorem, we obtain the sharp criteria on the existence
and the nonexistence of solutions to (\ref{1.1s}):

\begin{theorem} \label{th1.1}
Assume $k \in [1,\frac{n}{2})$ is an integer.

(1) The $2k$-order equation
\begin{equation} \label{1.1k}
(-\Delta)^k u(x)=u^p(x), \quad x \in R^n
\end{equation}
has positive solutions if and only if
$p \geq \frac{n+2k}{n-2k}$.

(2) The $2k$-order system (\ref{1.1s})
has a pair of positive solutions $(u,v)$ if $\frac{1}{p+1}+
\frac{1}{q+1} \leq \frac{n-2k}{n}$.
\end{theorem}

\paragraph{Remark 1.1.}
\begin{enumerate}
\item When $k=1$, the first result is coincident with the result in
\cite{CGS}.

\item Part 2 of our Theorem \ref{th1.1} together with the nonexistence result of Souplet \cite{Souplet}
imply that: for $n \leq 4$, (\ref{1.1s}) has a pair of solutions
{\it if and only if} $\frac{1}{p+1}+\frac{1}{q+1} \leq \frac{n-2k}{n}$.

\item For $k=1$, the nonexistence of solutions to (\ref{1.1s}), known as the
Lane-Emden conjecture, is still open for $n \geq 5$ (cf. \cite{Souplet}).
In the non-subcritical case, i.e. $\frac{1}{p+1}+\frac{1}{q+1} \leq
\frac{n-2}{n}$, (\ref{1.1s}) with $k=1$ has the positive solutions
(cf. \cite{SZ-1998}).
\end{enumerate}

\vskip 3mm

Next, we consider the HLS type system of nonlinear equations (\ref{1.3s})
and its scalar case $v \equiv u$, $q=p$:
\begin{equation} \label{1.3}
u(x)=\int_{R^n}\frac{u^p(y)dy}{|x-y|^{n-\alpha}}.
\end{equation}

Such equations are related to the study of the best constant of
Hardy--Littlewood-Soblev (HLS) inequality. Lieb \cite{Lieb}
classified all the extremal solutions of (\ref{1.3}), and thus
obtained the best constant in the HLS inequalities. He posed the
classification of all the solutions of (\ref{1.3}) as an open
problem.

The corresponding PDE is the semilinear equation
involving a fractional order differential operator
\begin{equation}
(- \Delta)^{\alpha/2} u = u^{(n+\alpha)/(n-\alpha)} , \;\; u>0,
\;\; \mbox{ in } R^n \label{de}.
\end{equation}
The classification of the solutions of (\ref{de}) with $\alpha=2$
has provided an important ingredient in the study of the
prescribing scalar curvature problem. It is also essential
in deriving priori estimates in many related nonlinear elliptic
equations. It was well studied by Gidas, Ni, and Nirenberg
\cite{GNN}. They proved that all the positive solutions
with reasonable behavior at infinity, namely
\begin{equation}
u(x) = O(\frac{1}{|x|^{n-2}}) \label{p1.3}
\end{equation}
are radially symmetric about some point. Caffarelli, Gidas, and
Spruck removed the decay condition (\ref{p1.3}) and obtained the
same result (cf. \cite{CGS}). Then Chen and Li \cite{CL1}, and Li
\cite{Li} simplified their proofs. Later, Chang, Yang and Lin
also considered some higher order equations (cf. \cite{CY},
\cite{Lin}). Wei and Xu \cite{WX}
generalized this result to the solutions of more general equation
(\ref{de}) with $\alpha$ being any even numbers between $0$ and
$n$. Chen, Li, and Ou solved the open problem as stated for the
integral equation (\ref{1.3}) or the corresponding PDE (\ref{de})
in \cite{ChLO}. The unique class of solutions can assume the form
\begin{equation} \label{poi}
u(x)=c(\frac{t}{t^2+|x|^2})^{\frac{n-\alpha}{2}}.
\end{equation}
Other related work
can be seen in \cite{CL}, \cite{YLi} and \cite{YiLi}.

Chen, Li and Ou \cite{CLO-CPDE} introduced the method of moving
planes in integral forms to study the symmetry of the solutions
for the HLS type system (\ref{1.3s}). Jin-Li and Hang thoroughly
discussed the regularity of the solutions of (\ref{1.3s}) (cf.
\cite{Hang} and \cite{JL}). They found the optimal integrability
intervals and established the smoothness for the
integrable solutions. Based on the results, \cite{LLM-CV} gave the
asymptotic behavior of the integrable solutions when $|x| \to 0$
and $|x| \to \infty$. Some Liouville type results can be seen in \cite{CDM}
and \cite{CL-DCDS}.

Another significance of the work \cite{CLO-CPDE} is the
equivalence of the integral equations and the PDEs involving the
fractional order differential operator. Recently, the fractional
Laplacians were applied extensively to describe various physical
and finance phenomena, such as anomalous diffusion, turbulence and
water waves, molecular dynamics, relativistic quantum mechanics,
and stable Levy process. The equivalence provides a technique in
studying the PDEs: one can use the corresponding integral
equations to investigate the global properties for those
phenomena.

A positive solution $u$ of (\ref{1.3}) is called a {\it finite
energy solution}, if $u \in L^{p+1}(R^n)$. Similarly, positive
solutions $u,v$ are called finite energy solutions of
(\ref{1.3s}), if $u \in L^{p+1}(R^n)$, $v \in L^{q+1}(R^n)$. Now,
we point out the relation between the critical conditions and the
existence of finite energy solutions of (\ref{1.3}) and
(\ref{1.3s}).

\begin{theorem} \label{th1.2}
(1) The HLS type integral equation (\ref{1.3})
has positive solutions in $L^{p+1}(R^n)$ if and only if
\begin{equation} \label{1.3ce}
p=\frac{n+\alpha}{n-\alpha}.
\end{equation}

(2) The system (\ref{1.3s})
has a pair of positive solutions $(u,v)$ in $L^{p+1}(R^n)
\times L^{q+1}(R^n)$ if and only if
\begin{equation} \label{1.3sce}
\frac{1}{p+1}+
\frac{1}{q+1}=\frac{n-\alpha}{n}.
\end{equation}
\end{theorem}

\begin{corollary} \label{coro1.3}
Let $k \in [1,n/2)$ be an integer.

(1) Assume $p>1$. The $2k$-order PDE (\ref{1.1k})
has positive solutions in $L^{p+1}(R^n)$ if and only if
\begin{equation} \label{1.3kce}
p=\frac{n+2k}{n-2k}.
\end{equation}

(2) Assume $pq>1$. System (\ref{1.1s})
has a pair of positive solutions $(u,v)$ in $L^{p+1}(R^n)
\times L^{q+1}(R^n)$ if and only if
\begin{equation} \label{1.3ksce}
\frac{1}{p+1}+
\frac{1}{q+1}=\frac{n-2k}{n}.
\end{equation}
\end{corollary}

\paragraph{Remark 1.2.}
\begin{enumerate}
\item In the subcritical case $p<\frac{n+\alpha}{n-\alpha}$, Theorem
3 in \cite{CLO4} shows that (\ref{1.3}) has no locally finite
energy solution by using the method of moving planes and the
Kelvin transformation. For system (\ref{1.3s}), the proof of nonexistence
in the subcritical case
$\frac{1}{p+1}+\frac{1}{q+1}>\frac{n-\alpha}{n}$ is rather
difficult. It is usually called the HLS conjecture (cf. \cite{CDM}
and \cite{CL-DCDS}). Partial results are known.

(i) If $p \leq \frac{\alpha}{n-\alpha}$ or $q \leq
\frac{\alpha}{n-\alpha}$, (\ref{1.3s}) has no any positive
solution. In addition, if $p=1$ or $q=1$, then (\ref{1.3s}) has no
any positive solution. If the subcritical condition
$\frac{1}{p+1}+\frac{1}{q+1}>\frac{n-2}{n}$ holds, then
(\ref{1.1s}) with $k=1$ has no locally bounded positive solution
$(u,v)$ in $L^{p_1}(R^n) \times L^{q_1}(R^n)$, where
$p_1=\frac{n(pq-1)}{2(q+1)}$, $q_1=\frac{n(pq-1)}{2(p+1)}$ (cf.
\cite{CL-DCDS}).

(ii) If $\alpha \in [2,n)$, (\ref{1.3s}) has no any radial
positive solution (cf. \cite{CDM}). If $pq>1$, (\ref{1.1s}) has no
any radial positive solution (cf. \cite{LGZ}).

\item In the supercritical case $p>\frac{n+\alpha}{n-\alpha}$ with
$\alpha=2$, Li, Ni and Serrin proved the semilinear Lane-Emden
equation (\ref{1.1}) has the decay solution (cf. \cite{YiLi},
\cite{NS} and \cite{PS}). According to Corollary \ref{coro1.3},
the energy of those solutions $u$ are infinite. Namely,
$\|u\|_{p+1}=\infty$. Similarly, the positive solutions $u,v$
obtained in \cite{SZ-1998} when $\frac{1}{p+1}+\frac{1}{q+1} <
\frac{n-2}{n}$ are not the finite energy solutions (i.e.
$\|u\|_{p+1}=\|v\|_{q+1}=\infty$).

\item Theorem \ref{th1.2} shows that the critical conditions are
the sufficient and necessary conditions of existences of the
finite energy solutions for the HLS type integral equation and the
system. On the other hand, if the critical conditions hold, we
want to know whether all the positive classical solutions $u,v$
are finite energy solutions. Namely, $\|u\|_{p+1}, \|v\|_{q+1} <
\infty$. For the scalar equation (\ref{1.3}) with the critical
case $p=\frac{n+\alpha}{n-\alpha}$, (\ref{poi}) is the unique
class of finite energy solutions of (\ref{1.3}) (cf. \cite{ChLO}).
For system (\ref{1.3s}), it is still open.
\end{enumerate}

\vskip 3mm

The following $\gamma$-Laplace equation is also concerned in this
paper
\begin{equation} \label{1.5}
-\Delta_\gamma u(x):=-div(|\nabla u|^{\gamma-2}\nabla u)=u^p(x), \quad x \in R^n.
\end{equation}

\begin{theorem} \label{th1.4}
The $\gamma$-Laplace equation (\ref{1.5})
has positive classical solutions with $\int_{R^n}|\nabla u|^\gamma dx<\infty$
if and only if $p=\gamma^*-1$, where $\gamma^*=\frac{n\gamma}{n-\gamma}$.
\end{theorem}

\paragraph{Remark 1.3.}
Serrin and Zou \cite{SZ} proved (\ref{1.5}) has the classical solution $u$
if and only if $p \geq \gamma^*-1$. Furthermore, Theorem \ref{th1.4} shows that
$u$ is also a finite energy solution (i.e. $\nabla u \in L^\gamma(R^n)$, see
Theorem \ref{th4.7}) if $p=\gamma^*-1$, and
$u$ is a infinite energy solution if $p>\gamma^*-1$.

\vskip 3mm

To study the $\gamma$-Laplace equations, we introduce
the Wolff potential of a positive locally integrable function $f$
$$
W_{\beta,\gamma}(f)(x):=\int_0^\infty
[\frac{\int_{B_t(x)}f(y)dy}{t^{n-\beta\gamma}}]^{\frac{1}{\gamma-1}} \frac{dt}{t}.
$$
The integral equation involving the Wolff potential
\begin{equation} \label{1.4c}
u(x)=c(x)W_{\beta,\gamma}(u^p)(x)
\end{equation}
is related with the study of many nonlinear problems.
The Wolff potentials are helpful to understand the nonlinear PDEs
such as the $\gamma$-Laplace equation and the $k$-Hessian equation
(cf. \cite{La}, \cite{LL}, \cite{LLM} and \cite{PV}).
According to \cite{KM}, if $\inf_{R^n}u=0$, there exists $C>0$ such that
the positive solution $u$ of (\ref{1.5}) satisfies
$$
\frac{1}{C}W_{1,\gamma}(u^p)(x) \leq u(x) \leq CW_{1,\gamma}(u^p)(x),
\quad x \in R^n.
$$
Thus, $u$ solves (\ref{1.4c}) for some double bounded $c(x)$.
Here, a function $c(x)$ is called {\it double bounded},
if there exists a positive constant
$C>0$ such that
$$
\frac{1}{C} \leq c(x) \leq C, \quad \forall ~x \in R^n.
$$
For the coupling system
$$
 \left \{
   \begin{array}{l}
      u(x)=W_{\beta,\gamma}(v^q)(x)\\
      v(x)=W_{\beta,\gamma}(u^p)(x),
   \end{array}
   \right.      
$$
Chen and Li \cite{ChenLi} proved the radial symmetry for the integrable solutions.
Afterward, Ma, Chen and Li \cite{ChLM} used the regularity lifting lemmas to
obtain the optimal integrability and the Lipschitz continuity. Based on these results,
\cite{Lei} obtained the decay rates of the integrable solutions when $|x|
\to \infty$.

The critical exponents and the critical conditions play a key role under the scaling
transform. In the following, some interesting observations are listed.

{\it
Under the scaling transform $u_\mu(x)=\mu^\sigma u(\mu x)$,

(1) the HLS equation (\ref{1.3}) and the energy $\|u\|_{p+1}$ are invariant
if and only if $p=\frac{n+\alpha}{n-\alpha}$;

(2) the Wolff equation
\begin{equation} \label{1.4}
u(x)=W_{\beta,\gamma}(u^p)(x)
\end{equation}
and the energy $\|u\|_{p+\gamma-1}$ are invariant if and only if
$p= \frac{n+\beta\gamma}{n-\beta\gamma}(\gamma-1)$; (\ref{1.4})
and the energy $\|u\|_{p+1}$ are invariant if and only if $p=
\gamma^*-1$ with $\gamma^*=\frac{n\gamma}{n-\beta\gamma}$;

(3) the $\gamma$-Laplace equation (\ref{1.5}) and the energy
$\|u\|_{p+\gamma-1}$ are invariant if and only if
$p=\frac{n+\gamma}{n-\gamma}(\gamma-1)$; (\ref{1.5}) and the
energy $\|u\|_{p+1}$ are invariant if and only if $p= \gamma^*-1$
with $\gamma^*=\frac{n\gamma}{n-\gamma}$;

(4) the HLS system (\ref{1.3s}) and $\|u\|_{p+1}$, $\|v\|_{q+1}$ are invariant
if and only if $\frac{1}{p+1}+
\frac{1}{q+1}=\frac{n-\alpha}{n}$;

(5) the Wolff system (\ref{1.4s})
and $\|u\|_{p+\gamma-1}$, $\|v\|_{q+\gamma-1}$ are invariant if
and only if
$\frac{1}{p+\gamma-1}+\frac{1}{q+\gamma-1}=\frac{n-\beta\gamma}
{n(\gamma-1)}$; (\ref{1.4s}) and $\|u\|_{p+1}$, $\|v\|_{q+1}$ are
invariant if and only if $p=q$ or $\gamma=2$;

(6) the $\gamma$-Laplace system (\ref{1.5s})
and $\|u\|_{p+\gamma-1}$, $\|v\|_{q+\gamma-1}$ are invariant if
and only if
$\frac{1}{p+\gamma-1}+\frac{1}{q+\gamma-1}=\frac{n-\gamma}
{n(\gamma-1)}$; (\ref{1.5s}) and $\|u\|_{p+1}$, $\|v\|_{q+1}$ are
invariant if and only if $p=q$ or $\gamma=2$. }

\paragraph{Remark 1.4.} Here an interesting observation is, the critical exponent
$\frac{n+\gamma}{n-\gamma}(\gamma-1)$ is different from the divided exponent
$\gamma^*-1$ in Theorem \ref{th1.4}
except $\gamma=2$. The reason is that those critical numbers
are in the different finite energy functions classes $L^{p+\gamma-1}(R^n)$ and
$L^{p+1}(R^n)$, respectively.

\vskip 3mm

Next, we are concerned with the sufficient and necessary
conditions for the existence of the positive solutions of
equations and systems with some double bounded
coefficients. Here, new divided numbers and conditions
appear.

\begin{theorem} \label{th1.5}
(1) The equation
\begin{equation} \label{1.3c}
u(x)=c(x)\int_{R^n}\frac{u^p(y)dy}{|x-y|^{n-\alpha}}
\end{equation}
has positive solutions for some
double bounded $c(x)$ if and only if $p > \frac{n}{n-\alpha}$.

(2) The HLS system
\begin{equation}
 \left \{
   \begin{array}{l}
      u(x)=c_1(x) \displaystyle\int_{R^n}\frac{v^q(y)dy}{|x-y|^{n-\alpha}}\\
      v(x)=c_2(x)\displaystyle\int_{R^n}\frac{u^p(y)dy}{|x-y|^{n-\alpha}}.
   \end{array}
   \right.\label{1.3cs}          
 \end{equation}
has positive solutions $u,v$ for
some double bounded $c_1(x)$ and $c_2(x)$, if and only if $pq>1$ and
$\max\{\frac{\alpha(p+1)}{pq-1}, \frac{\alpha(q+1)}{pq-1}\} <
n-\alpha$.
\end{theorem}

\begin{corollary} \label{coro1.6}
Let $k \in [1,n/2)$ be an integer.

(1) Assume $p> 1$. The $2k$-order PDE
\begin{equation} \label{1.2kc}
(-\Delta)^k u(x)=c(x)u^p(x), \quad x \in R^n,
\end{equation}
has positive solutions for some
double bounded $c(x)$ if and only if $p > \frac{n}{n-2k}$.

(2) Assume $pq>1$. The system
\begin{equation}
 \left \{
   \begin{array}{l}
      (-\Delta)^k u(x)=c_1(x)v^q(x)\\
      (-\Delta)^k v(x)=c_2(x)u^p(x).
   \end{array}
   \right.    \label{1.1cs}      
 \end{equation}
has positive solutions $u,v$ for some double bounded $c_1(x)$ and
$c_2(x)$, if and only if $\max\{\frac{2k(p+1)}{pq-1},
\frac{2k(q+1)}{pq-1}\} < n-2k$.
\end{corollary}

\begin{theorem} \label{th1.7}

(1) The equation (\ref{1.4c})
$$
u(x)=c(x)W_{\beta,\gamma}(u^p)(x)
$$
has positive solutions for some double bounded $c(x)$, if and only if
$$
p>\frac{n(\gamma-1)}{n-\beta\gamma}.
$$

(2) The system
\begin{equation}
 \left \{
   \begin{array}{l}
      u(x)=c_1(x)W_{\beta,\gamma}(v^q)(x)\\
      v(x)=c_2(x)W_{\beta,\gamma}(u^p)(x).
   \end{array}
   \right.\label{1.4cs}          
 \end{equation}
has positive solutions $u,v$ for some double bounded $c_1(x)$ and $c_2(x)$,
if and only if $pq>(\gamma-1)^2$ and
$$\max\{\frac{\beta\gamma(p+\gamma-1)}{pq-(\gamma-1)^2},
\frac{\beta\gamma(q+\gamma-1)}{pq-(\gamma-1)^2}\} <
\frac{n-\beta\gamma}{\gamma-1}.
$$
\end{theorem}

\begin{corollary} \label{coro1.8}
(1) If $p>\frac{n(\gamma-1)}{n-\gamma}$, then
\begin{equation} \label{1.5c}
-\Delta_\gamma u(x)=c(x)u^p(x), \quad x \in R^n
\end{equation}
has positive solutions for some double bounded $c(x)$.
If $0<p \leq \frac{n(\gamma-1)}{n-\gamma}$, then for any double bounded
$c(x)$, (\ref{1.5c}) has no
positive solution satisfying $\inf_{R^n}u=0$.

(2) If $pq>(\gamma-1)^2$ and
$\max\{\frac{\gamma(q+\gamma-1)}{pq-(\gamma-1)^2},
\frac{\gamma(p+\gamma-1)}{pq-(\gamma-1)^2}\}<\frac{n-\gamma}{\gamma-1}$,
then there exist positive solutions $u,v$ of the $\gamma$-Laplace
system
\begin{equation}
 \left \{
   \begin{array}{l}
      -\Delta_\gamma u(x)=c_1(x)v^q(x), \quad x \in R^n,\\
      -\Delta_\gamma v(x)=c_1(x)u^p(x), \quad x \in R^n
   \end{array}
   \right.   \label{1.5cs}     
 \end{equation}
for some double bounded $c_1(x)$ and $c_2(x)$. On the contrary,
for any double bounded functions $c_1(x)$ and $c_2(x)$, if one of the
following conditions holds

(i) $0<pq \leq (\gamma-1)^2$;

(ii) $pq>(\gamma-1)^2$ and
$$
\max\{\frac{\gamma(q+\gamma-1)}{pq-(\gamma-1)^2},
\frac{\gamma(p+\gamma-1)}{pq-(\gamma-1)^2}\} \geq \frac{n-\gamma}{\gamma-1}.
$$
Then (\ref{1.5cs}) has no positive solutions $u,v$ satisfying
$\inf_{R^n}u=\inf_{R^n}v=0$.
\end{corollary}

\paragraph{Remark 1.5.}
Comparing with Theorem \ref{th1.2}-Theorem \ref{th1.4},
we obtain, from Theorem \ref{th1.5}-Corollary \ref{coro1.8},
other divided conditions on the existence of the positive solutions
of the equations and systems with ratio coefficients $c(x)$, $c_1(x)$ and $c_2(x)$.
These divided conditions are called {\it the secondary critical conditions}.
The secondary critical conditions are more relaxed than those in Theorem \ref{th1.2}
-Theorem \ref{th1.4} because the solutions classes of
the equations and systems with ratio coefficients
are larger than that in the case of $c(x)\equiv Constant$.

\vskip 3mm

In the proofs of Theorem \ref{th1.5}-Corollary \ref{coro1.8}, we apply a special
iteration scheme and some critical asymptotic analysis to establish the existence
and the nonexistence, and hence obtain the sharp criteria.

The contents of this paper are as follows. In Section 2, we prove Theorem \ref{th1.5} (1),
Corollary \ref{coro1.6} (1), Theorem \ref{th1.7} (1) and Corollary \ref{coro1.8} (1).
Theorem \ref{th1.5} (2), Corollary \ref{coro1.6} (2), Theorem \ref{th1.7} (2)
and Corollary \ref{coro1.8} (2) are proved
in Section 3. In Section 4.2, we prove (1) of Theorem \ref{th1.2}, which
covers (1) of Corollary \ref{coro1.3}. The proof of of Theorem \ref{th1.4}
is given in Section 4.3. In Section 5, we give the
proofs of (2) of Theorem \ref{th1.2} and (2) of Corollary \ref{coro1.3}.
The argument on Theorem \ref{th1.1} is given in Sections 6 (see Remark 6.1).

\section{Equations with variable coefficients}

The following proposition is often used in this paper.

\begin{proposition} \label{prop2.1}
If $w \in L^1(R^n)$, then we can find $R_j \to \infty$ such that
$$
R_j\int_{\partial B_{R_j}(0)}|w|ds \to 0 \quad and \quad
R_j^n\int_{S^{n-1}}|w|ds \to 0.
$$
Here $S^{n-1}=\partial B_1(0)$.
\end{proposition}

\begin{proof}
In view of $\|w\|_{L^1(R^n)}<\infty$,
it follows from the definition of the improper integral that
$$
\lim_{R \to \infty}\int_{B_{2R}(0) \setminus B_R(0)}|w|dx =0.
$$
Hence, as $R \to \infty$,
$$
\inf_{[R,2R]}(r\int_{\partial B_r(0)}|w|ds) \to 0 \quad
and \quad \inf_{[R,2R]}(r^n\int_{S^{n-1}}|w|ds) \to 0.
$$
There exist $R_j \in [R,2R]$, such that as $R_j \to \infty$,
$$
R_j\int_{\partial B_{R_j}(0)}|w|ds \to 0 \quad and \quad
R_j^n\int_{S^{n-1}}|w|ds \to 0.
$$
\end{proof}

\subsection{HLS type integral equation}

In this subsection, we give a relation between the exponents and the
existence of positive solutions for integral equations involving the
Riesz potentials. First we consider the semilinear Lane-Emden type equations
\begin{equation}
-\Delta u(x)=c(x)u^p(x), \quad x \in R^n.
\label{LE}
\end{equation}

\begin{theorem} \label{th2.2}
Let $p \geq 1$. Then
(\ref{LE}) has a positive solution for some double bounded $c(x)$,
if and only if
$p > \frac{n}{n-2}$.
\end{theorem}

\begin{proof}
{\it Step 1.}
If $p > \frac{n}{n-2}$,
we claim that (\ref{LE}) has the special solution as follows
\begin{equation} \label{EXP}
u(x)=\frac{1}{(1+|x|^2)^\theta},
\end{equation}
where $\theta>0$ will be determined later.

Denote $|x|$ by $r$, and set $U(r)=U(|x|)=u(x)$.
By a simply calculation, we obtain
\begin{equation} \label{2.0}
-\Delta u=-U_{rr}-\frac{n-1}{r}U_r
=\frac{2\theta}{(1+r^2)^{\theta+1}}
\frac{(n-2-2\theta)r^2+n}{1+r^2}.
\end{equation}

Take $\theta=\frac{1}{p-1}$. Then $n-2-2\theta>0$ and
$$
-\Delta u=\frac{c(r)}{(1+r^2)^{\theta+1}}=c(r)u^{1+1/\theta}=c(r)u^p.
$$
Namely, (\ref{EXP}) with the slow rate $2\theta=\frac{2}{p-1}$
is a solution for some double bounded $c(r)$.

Moreover, if $p=\frac{n+2}{n-2}$, there also exists a fast
decaying solution with rate $2\theta=n-2$. Now,
$$
-\Delta u=\frac{c(r)}{(1+r^2)^{\theta+2}}=c(r)u^{1+2/\theta}=c(r)u^p.
$$
Namely, (\ref{EXP}) with the fast rate $2\theta=n-2$
is a solution for some double bounded $c(r)$.

{\it Step 2.} We prove (\ref{LE}) has no positive solution when
$1 \leq p \leq \frac{n}{n-2}$.

Otherwise, let $u$ be a positive solution. Take $x_0 \in R^n$ and
denote $B_R(x_0)$ by $B$. Let
$$
\phi(x)=\phi_R(x)=c_R(\frac{1}{|x-x_0|^{n-2}} -\frac{1}{R^{n-2}}),
$$
where $0<c_R \to c_* \in (0,\infty)$ as $R \to \infty$. Then,
$\phi$ solves
$$
\left \{
   \begin{array}{l}
      -\Delta \phi(x)=\delta(x), \quad x \in B,\\
      \phi=0, \quad on ~\partial B.
      \end{array}
   \right.
$$
Here $\delta$ is a Dirac function at $x_0$. Then,
\begin{equation} \label{ghj}
\begin{array}{ll}
&\displaystyle\int_B c(x)u^p(x)\phi(x) dx=-\int_B\phi \Delta udx
=\int_B\nabla u \nabla \phi dx\\[3mm]
&=\displaystyle\int_{\partial B}u\partial_{\nu}\phi ds-\int_Bu\Delta\phi dx
=\displaystyle\int_{\partial B}u\partial_{\nu}\phi ds+u(x_0).
\end{array}
\end{equation}
Here $\nu$ is the unit outward normal vector on $\partial B$.
Noting $\partial_{\nu} \phi<0$, we have
$$
\int_B u^p\phi dx \leq cu(x_0)<\infty.
$$
Let $R \to \infty$, there holds
$$
\int_{R^n}\frac{u^p(y)dy}{|x_0-y|^{n-2}}<\infty.
$$
According to Proposition \ref{prop2.1},
there exists $R_j \to \infty$ (we still denote it by $R$)
such that
$$
R\int_{\partial B}\frac{u^p(y)ds}{R^{n-2}} \to 0.
$$
Thus, noting $p \geq 1$, we can use the H\"older inequality to deduce
that
$$
|\int_{\partial B}u \partial_{\nu} \phi ds|
\leq \frac{c}{R^{n-1}}\int_{\partial B}uds \leq
\frac{c}{R^{2/p}}(R\int_{\partial B}\frac{u^p(y)ds}{R^{n-2}})^{1/p}
\to 0,
$$
when $R \to \infty$. Let $R \to \infty$ in (\ref{ghj}), then
$$
u(x_0)=c\int_{R^n}\frac{c(y)u^p(y)dy}{|x_0-y|^{n-2}}.
$$

Write $w(x)=c^{1/p}(x)u(x)$, then $w$ solves
the integral equation involving the Newton potential
$$
w(x)=c^{1/p}(x)\int_{R^n}\frac{w^p(y)dy}{|x-y|^{n-2}}.
$$
However, this integral equation has no positive solution for
any double bounded $c$ when $0<p \leq \frac{n}{n-2}$.
The proof is a special case of the corresponding proof of
Theorem \ref{th2.3}, which handles a more general integral
equation involving the Riesz potential.
\end{proof}

\begin{theorem} \label{th2.3}
The HLS type integral equation
\begin{equation}
u(x)=c(x)\int_{R^n}\frac{u^p(y)dy}{|x-y|^{n-\alpha}}
\label{HLS}
\end{equation}
has a positive solution for some double bounded $c(x)$,
if and only if
\begin{equation} \label{cc1}
p > \frac{n}{n-\alpha}.
\end{equation}
\end{theorem}

\begin{proof}
When $|x| \leq 2R$ for some $R>0$, $u(x)$ is proportional to $\int_{R^n}
|x-y|^{\alpha-n}u^p(y)dy$. Thus, we only consider the case of $|x|>2R$.

{\it Step 1.}
Inserting (\ref{EXP}) into the right hand side of (\ref{HLS}), we can find
some double bounded function $c(x)$ such that as $|x|>2R$ for some $R>0$,
$$\begin{array}{ll}
\displaystyle\int_{R^n}\frac{u^p(y)dy}{|x-y|^{n-\alpha}}
&=\displaystyle\frac{c(x)}{(1+|x|^2)^{(n-\alpha)/2}}\int_{B_R(0)}\frac{dy}{(1+|y|^2)^{p\theta}}\\[3mm]
&\quad +\displaystyle\frac{c(x)}{(1+|x|^2)^{p\theta}}\int_{B_{|x|/2}(x)}\frac{dy}{|x-y|^{n-\alpha}}\\[3mm]
&\quad +c(x)\displaystyle\int_{B_R^c(0)\setminus B_{|x|/2}(x)}\frac{dy}{|x-y|^{n-\alpha}|y|^{2p\theta}}.
\end{array}
$$
If $p> \frac{n}{n-\alpha}$, we take $2\theta=\frac{\alpha}{p-1}$
and hence $\alpha<2p\theta < n$. Then,
$$\begin{array}{ll}
&\quad \displaystyle\int_{R^n}\frac{u^p(y)dy}{|x-y|^{n-\alpha}}\\[3mm]
&=\displaystyle\frac{c(x)}{(1+|x|^2)^{(n-\alpha)/2}}
+\frac{c(x)}{(1+|x|^2)^{p\theta-\alpha/2}}+c(x)\int_{|x|/2}^\infty r^{n-(n-\alpha+2p\theta)}
\frac{dr}{r}\\[3mm]
&=\displaystyle\frac{c(x)}{(1+|x|^2)^{p\theta-\alpha/2}}=c(x)u(x)
\end{array}
$$
for some double bounded function $c(x)$. This result shows that
(\ref{HLS}) has the slowly decaying radial solution as (\ref{EXP}).

Moreover, we can also find a fast decaying solution.
Now, take $2\theta=n-\alpha$, then $2p\theta > n$ as long as
$p > \frac{n}{n-\alpha}$. Thus,
$$\begin{array}{ll}
&\quad \displaystyle\int_{R^n}\frac{u^p(y)dy}{|x-y|^{n-\alpha}}\\[3mm]
&=\displaystyle\frac{c(x)}{(1+|x|^2)^{(n-\alpha)/2}}
+\frac{c(x)}{(1+|x|^2)^{p\theta-\alpha/2}}+c(x)\int_{|x|/2}^\infty r^{n-(n-\alpha+2p\theta)}
\frac{dr}{r}\\[3mm]
&=\displaystyle\frac{c(x)}{(1+|x|^2)^{(n-\alpha)/2}}=c(x)u(x)
\end{array}
$$
for some double bounded function $c(x)$.

{\it Step 2.} We prove (\ref{HLS}) has no positive solution when $0<p< \frac{n}{n-\alpha}$.

Suppose $u$ is a positive solution, then it follows a
contradiction. In fact, when $|x|>R$ with $R>0$, $|x-y| \leq 2|x|$
for $y \in B_R(0)$. In addition, $\int_{B_R(0)}u^p(y)dy \geq c$.
Hence,
$$
u(x) \geq c|x|^{\alpha-n}\int_{B_R(0)}u^p(y)dy \geq
\frac{c}{|x|^{a_0}}, \quad for ~ |x|>R.
$$
Here $a_0=n-\alpha$. Using this estimate, for $|x|>R$ we also get
$$
u(x) \geq c\int_{B_{|x|/2}(x)}
\frac{|y|^{-pa_0}dy}{|x-y|^{n-\alpha}}
\geq \frac{c}{|x|^{pa_0-\alpha}}:=\frac{c}{|x|^{a_1}}.
$$
By induction, we can obtain
$$
u(x) \geq \frac{c}{|x|^{a_j}}, \quad for ~ |x|>R,
$$
where $j=0,1,\cdots$, and
$$
a_j=pa_{j-1}-\alpha.
$$
In view of $0<p<\frac{n}{n-\alpha}$, we claim that $\{a_j\}$ is decreasing. In fact,
$a_1-a_0=(p-1)a_0-\alpha=(p-1)(n-\alpha)-\alpha<0$. Suppose $a_k<a_{k-1}$ for $k=1,2,\cdots,j$, then
$$
a_{j+1}-a_j=(p-1)a_j-\alpha<(p-1)a_0-\alpha=(p-1)(n-\alpha)-\alpha<0.
$$
This induction shows our claim.

Next we claim that there exists $j_0$ such that
\begin{equation}
a_{j_0}<0.
\label{2.8}
\end{equation}
Once it is verified, then
$$
u(x) \geq c\int_{R^n \setminus
B_R(0)}\frac{|y|^{-pa_{j_0}}dy}{|x-y|^{n-\alpha}} =\infty.
$$
It is impossible.

{\it Proof of (\ref{2.8}).}
In fact,
$$\begin{array}{ll}
a_j&=pa_{j-1}-\alpha=p(pa_{j-2}-\alpha)-\alpha=\cdots\\[3mm]
&=p^ja_0-\alpha(p^{j-1}+p^{j-2}+\cdots+p+1).
\end{array}
$$

When $p \in (1,\frac{n}{n-\alpha})$,
$$
a_j=p^j(n-\alpha)-\alpha\frac{p^j-1}{p-1}=(n-\alpha-\frac{\alpha}{p-1})p^j+\frac{\alpha}{p-1}.
$$
By virtue of $p < \frac{n}{n-\alpha}$, $n-\alpha-\frac{\alpha}{p-1}<0$, we can find a suitably large
$j_0$ such that $a_{j_0}<0$.

When $p=1$, $a_j=a_0-\alpha j$. Thus, $a_{j_0}<0$ for some suitably large $j_0$.

When $p \in (0,1)$, let $j \to \infty$. Then
$$
a_j=p^ja_0-\alpha\frac{1-p^j}{1-p} \to -\frac{\alpha}{1-p}<0.
$$
This implies $a_{j_0}<0$ for some $j_0$.

Thus, (\ref{2.8}) is verified.

{\it Step 3.} We prove (\ref{HLS}) has no positive solution when $p=\frac{n}{n-\alpha}$.

Otherwise, $u$ is a positive solution.
For $R>0$, denote $B_R(0)$ by $B$. From (\ref{HLS}) it follows that
\begin{equation} \label{low1}
u(x) \geq \frac{1}{(R+|x|)^{n-\alpha}} \int_B u^p(y)dy.
\end{equation}
Thus, taking $p$ powers of (\ref{low1}) and integrating on $B$, we have
\begin{equation} \label{low2}
\begin{array}{ll}
\displaystyle\int_B  u^p(x)dx &\geq \displaystyle\int_B \frac{
dx}{(R+|x|)^{n}} (\int_B u^p(y)dy)^p\\[3mm] &\geq
c(\displaystyle\int_B u^p(y)dy)^p.
\end{array}
\end{equation}
Here $c$ is independent of $R$. Letting $R \to \infty$, we see
$u \in L^p(R^n)$.

Taking $p$ powers of (\ref{low1}) and integrating on
$A_R:=B_{2R}(0) \setminus B_R(0)$, we get
$$
\int_{A_R} u^p(x)dx \geq \int_{A_R}\frac{
dx}{(R+|x|)^{n}} (\int_B u^p(y)dy)^p \geq
c(\displaystyle\int_B u^p(y)dy)^p.
$$
Letting $R \to \infty$, and noting
$u \in L^p(R^n)$, we obtain
$$
\int_{R^n} u^p(y)dy=0,
$$
which contradicts with $u>0$.
\end{proof}

\begin{corollary} \label{coro2.4}
Assume $k \in [1,n/2)$ and $p>1$.
The higher order semilinear PDE
\begin{equation} \label{highorder}
(-\Delta)^k u(x)=c(x)u^p(x), \quad x \in R^n,
\end{equation}
has a positive solution for some double bounded $c(x)$,
if and only if $p > \frac{n}{n-2k}$.
\end{corollary}

\begin{proof}
If $u>0$ solves the integral equation (\ref{HLS}) with $\alpha=2k$,
it is easy to see
that $u$ also solves the higher order semilinear PDE
(\ref{highorder}).
On the contrary, if $p>1$ and $u$ solves
(\ref{highorder}), \cite{LGZ} proved $(-\Delta)^i u>0$ for $i=1,2,\cdots,
k-1$. Similar to the argument in \cite{ChLO}, (\ref{highorder})
is equivalent to (\ref{HLS}) with $\alpha=2k$. Therefore, if $p>1$,
Theorem \ref{th2.3} shows that (\ref{highorder})
has positive solutions for some double bounded function $c(x)$, if and only if
$p > \frac{n}{n-2k}$.
\end{proof}

\subsection{Integral equation involving the Wolff potential}

\begin{theorem} \label{th2.5}
The Wolff type integral equation
\begin{equation} \label{wolff}
u(x)=c(x)W_{\beta,\gamma}(u^p)(x)
\end{equation}
has a positive solution for some double bounded $c(x)$,
if and only if
$$
p > \frac{n(\gamma-1)}{n-\beta\gamma}.
$$
\end{theorem}

\begin{proof}
{\it Step 1.} Existence.

Inserting (\ref{EXP}) into $W_{\beta,\gamma}(u^p)(x)$,
we obtain
$$
W_{\beta,\gamma}(u^p)(x)=(\int_0^{|x|/2}+\int_{|x|/2}^\infty)
[\int_{B_t(x)}\frac{dy}{(1+|y|^2)^{p\theta}}t^{\beta\gamma-n}]
^{\frac{1}{\gamma-1}} \frac{dt}{t}:=I_1+I_2.
$$

When $|x| \leq R$ for some $R>0$, then $u$ is proportional to
$W_{\beta,\gamma}(u^p)$. So we also only consider suitably large $|x|$.

Clearly,
$$\begin{array}{ll}
I_1&=\displaystyle\int_0^{|x|/2}[\frac{\int_{B_t(x)}(1+|y|^2)^{-p\theta}dy}
{t^{n-\beta\gamma}}]^{\frac{1}{\gamma-1}}
\frac{dt}{t}\\[3mm]
&=c(1+|x|^2)^{-\frac{p\theta}{\gamma-1}}\displaystyle\int_0^{|x|/2} t^{\frac{\beta\gamma}{\gamma-1}}
\frac{dt}{t}\\[3mm]
&=c(1+|x|^2)^{-\frac{p\theta}{\gamma-1}}|x|^{\frac{\beta\gamma}{\gamma-1}}
=c(1+|x|^2)^{\frac{\beta\gamma-2p\theta}{2(\gamma-1)}}.
\end{array}
$$

Take the slow rate $2\theta=\frac{\beta\gamma}{p-\gamma+1}$.
Now, $\beta\gamma<2p\theta < n$ in view of $p > \frac{n(\gamma-1)}{n-\beta\gamma}$,
and hence
$$\begin{array}{ll}
I_2&=c\displaystyle\int_{|x|/2}^\infty (\frac{\int_{B_t(x)}(1+|y|^2)^{-p\theta}dy}{t^{n-\beta\gamma}}
)^{\frac{1}{\gamma-1}} \frac{dt}{t}\\[3mm]
&=c(x)\displaystyle\int_{|x|/2}^\infty (\frac{t^{n-2p\theta}}{t^{n-\beta\gamma}}
)^{\frac{1}{\gamma-1}} \frac{dt}{t}
=c(x)(1+|x|^2)^{\frac{\beta\gamma-2p\theta}{2(\gamma-1)}}.
\end{array}
$$
Thus, $I_1+I_2=c(x)u(x)$ for some double bounded $c(x)$.

Similarly, we also find a fast decaying solution. In fact, taking
$2\theta=\frac{n-\beta\gamma}{\gamma-1}$, we also have $2p\theta > n$
from $p > \frac{n(\gamma-1)}{n-\beta\gamma}$, and hence
$$\begin{array}{ll}
I_2&=\displaystyle\int_{|x|/2}^\infty (\frac{\int_{B_t(x)\cap B_1(0)}(1+|y|^2)^{-p\theta}dy
+\int_{B_t(x)\setminus B_1(0)}(1+|y|^2)^{-p\theta}dy}{t^{n-\beta\gamma}}
)^{\frac{1}{\gamma-1}} \frac{dt}{t}\\[3mm]
&=c(x)\displaystyle\int_{|x|/2}^\infty t^{-\frac{n-\beta\gamma}{\gamma-1}}
\frac{dt}{t}=c(x)(1+|x|^2)^{-\frac{n-\beta\gamma}{2(\gamma-1)}}.
\end{array}
$$
There also holds
$I_1+I_2=c(x)(1+|x|^2)^{-\frac{n-\beta\gamma}{2(\gamma-1)}}
=c(x)u(x)$.

{\it Step 2.} Nonexistence.

{\it Substep 2.1.}
Let
\begin{equation} \label{2.10}
0<p<\frac{n(\gamma-1)}{n-\beta\gamma}.
\end{equation}
Suppose that $u$ solves (\ref{wolff}), then
\begin{equation} \label{low-ini}
u(x) \geq c\int_{2|x|}^\infty t^{\frac{\beta\gamma-n}{\gamma-1}} \frac{dt}{t}
=\frac{c}{|x|^{a_0}},
\end{equation}
since $\int_{B_1(0)}u^p(y)dy \geq c$,
where $a_0=\frac{n-\beta\gamma}{\gamma-1}$.
By this estimate, we have
\begin{equation} \label{2.11}
u(x) \geq c\int_{2|x|}^\infty (\frac{\int_{B_{t-|x|}(0)}|y|^{-pa_0}dy}{
t^{n-\beta\gamma}})^{\frac{1}{\gamma-1}} \frac{dt}{t}
\geq c\int_{2|x|}^\infty (t^{\beta\gamma-pa_0})^{\frac{1}{\gamma-1}}
\frac{dt}{t}.
\end{equation}

When $\frac{p}{\gamma-1} \in (0,\frac{\beta\gamma}{n-\beta\gamma}]$,
we have $\beta\gamma-pa_0 \geq 0$. Eq. (\ref{2.11}) implies $u(x)=\infty$. It is impossible.

Next, we consider the case $\frac{p}{\gamma-1} \in (\frac{\beta\gamma}{n-\beta\gamma},
\frac{n}{n-\beta\gamma})$.
Now (\ref{2.11}) leads to
$$
u(x) \geq \frac{c}{|x|^{a_1}},
$$
where $a_1=\frac{p}{\gamma-1}a_0-\frac{\beta\gamma}{\gamma-1}$.

Write
\begin{equation} \label{2.14}
a_j=\frac{p}{\gamma-1}a_{j-1}-\frac{\beta\gamma}{\gamma-1},~j=1,2,\cdots.
\end{equation}
Suppose that $a_k<a_{k-1}$ for $k=1,2,\cdots,j-1$.
By virtue of (\ref{2.10}), it follows
$$\begin{array}{ll}
a_j-a_{j-1}&=(\frac{p}{\gamma-1}-1)a_{j-1}-\frac{\beta\gamma}{\gamma-1}
< (\frac{p}{\gamma-1}-1)a_0-\frac{\beta\gamma}{\gamma-1}\\[3mm]
&=(\frac{p}{\gamma-1}-1)\frac{n-\beta\gamma}{\gamma-1}-\frac{\beta\gamma}{\gamma-1}
=(\frac{n-\beta\gamma}{(\gamma-1)^2})p-\frac{n}{\gamma-1}\\[3mm]
&<(\frac{n-\beta\gamma}{(\gamma-1)^2})\frac{n(\gamma-1)}{n-\beta\gamma}-\frac{n}{\gamma-1}=0.
\end{array}
$$
Thus, $\{a_j\}_{j=0}^\infty$ is decreasing as long as (\ref{2.10}) is true.

Furthermore, we claim that there must be $j_0>0$ such that $a_{j_0} \leq 0$.
This leads to $u(x)=\infty$,
which contradicts with the fact that $u$ is a positive solution.

In fact, by (\ref{2.14}) we get
$$
a_j=(\frac{p}{\gamma-1})^ja_0-[1+\frac{p}{\gamma-1}+\cdots+(\frac{p}{\gamma-1})^{j-1}]
\frac{\beta\gamma}{\gamma-1}.
$$

If $\frac{p}{\gamma-1}=1$, then we can find a large $j_0$ such that
$$
a_{j_0}=a_0-j_0\frac{\beta\gamma}{\gamma-1}
\leq 0.
$$

If $\frac{p}{\gamma-1}\in (1,\frac{n}{n-\beta\gamma})$, then using
$a_0-\frac{\beta\gamma}{p-\gamma+1}<0$
which is implied by (\ref{2.10}), we can find a large $j_0$ such that
$$\begin{array}{ll}
a_{j_0}&=(\displaystyle\frac{p}{\gamma-1})^{j_0}a_0-\frac{(\frac{p}{\gamma-1})^{j_0}-1}
{\frac{p}{\gamma-1}-1}\frac{\beta\gamma}{\gamma-1}\\[3mm]
&=(\displaystyle\frac{p}{\gamma-1})^{j_0}(a_0-\frac{\beta\gamma}{p-\gamma+1})
+\frac{\beta\gamma}{p-\gamma+1} \leq 0.
\end{array}
$$

If $\frac{p}{\gamma-1} \in (0,1)$, letting $j \to \infty$, we get
$$
a_j=(\frac{p}{\gamma-1})^ja_0-\frac{1-(\frac{p}{\gamma-1})^j}
{1-\frac{p}{\gamma-1}}\frac{\beta\gamma}{\gamma-1} \to \frac{\beta\gamma}{p-\gamma+1}<0.
$$
Thus, there must be $j_0$ such that
$a_{j_0} \leq 0$.

{\it Substep 2.2.} Let $p=\frac{n(\gamma-1)}{n-\beta\gamma}$. By the argument of
Supstep 2.1, we can assume $p=\frac{n(\gamma-1)}{n-\beta\gamma}>1$ here.
We deduce the contradiction if $u$ is a positive solution of (\ref{wolff}).

For $R>0$, denote $B_R(0)$ by $B_R$. By using the H\"older
inequality, from (\ref{wolff}) we deduce that for any $x \in B_R$,
$$\begin{array}{ll}
&\quad \displaystyle\int_0^R \int_{B_t(x)}u^p(y)dy dt\\[3mm]
&\leq (\displaystyle\int_0^R
(\int_{B_t(x)}u^p(y)dy)^{\frac{1}{\gamma-1}}
t^{\frac{\beta\gamma-n}{\gamma-1}-1}dt)^{\gamma-1}
(\int_0^R t^{\frac{n-\beta\gamma+\gamma-1}{2-\gamma}}dt)^{2-\gamma}\\[3mm]
&=CR^{n-\beta\gamma+1} (\displaystyle\int_0^R
(\frac{\int_{B_t(x)}u^p(y)dy}{t^{n-\beta\gamma}})^{\frac{1}{\gamma-1}}
\frac{dt}{t})^{\gamma-1}.
\end{array}
$$
Hence, exchanging the order of the integral variables, we have
$$\begin{array}{ll}
u(x) &\geq c\displaystyle\int_0^R
(\frac{\int_{B_t(x)}u^p(y)dy}{t^{n-\beta\gamma}})^{\frac{1}{\gamma-1}}
\frac{dt}{t}\\[3mm] &\geq cR^{-\frac{n-\beta\gamma+1}{\gamma-1}}
(\displaystyle\int_0^R(\int_{B_t(x)}u^p(y)dy)dt)^{\frac{1}{\gamma-1}}\\[3mm]
&\geq cR^{-\frac{n-\beta\gamma+1}{\gamma-1}}
(\displaystyle\int_{B_R}u^p(y)(\int_{|x-y|}^R dt)dy)^{\frac{1}{\gamma-1}}\\[3mm]
&\geq cR^{-\frac{n-\beta\gamma}{\gamma-1}}
(\displaystyle\int_{B_{R/4}}u^p(y)dy)^{\frac{1}{\gamma-1}}.
\end{array}
$$
Therefore, we get
\begin{equation} \label{low1m}
u^p(x) \geq cR^{p\frac{\beta\gamma-n}{\gamma-1}}(\int_{B_{R/4}}
u^p(y)dy)^{\frac{p}{\gamma-1}}.
\end{equation}
Integrating on $B_{R/4}$ and using
$p=\frac{n(\gamma-1)}{n-\beta\gamma}$ again, we get
$$\begin{array}{ll}
&\quad \displaystyle\int_{B_{R/4}}u^p(x)dx \\[3mm]
&\geq cR^{p\frac{\beta\gamma-n}{\gamma-1}}
\displaystyle\int_{B_{R/4}}dx(\int_{B_{R/4}}u^p(y)dy)^{\frac{p}{\gamma-1}}\\[3mm]
&\geq c(\displaystyle\int_{B_{R/4}}u^p(y)dy)^{\frac{p}{\gamma-1}}.
\end{array}
$$
Here $c$ is independent of $R$. Letting $R \to \infty$ and noting
$p>\gamma-1$, we have
\begin{equation} \label{tianim}
\int_{R^n} u^p(x) dx<\infty.
\end{equation}

Integrating (\ref{low1m}) on $A_R=B_{R/4} \setminus B_{R/8}$ yields
$$
\int_{A_R} u^p(x)dx \geq c
R^{p\frac{\beta\gamma-n}{\gamma-1}}\int_{A_R}dx (\int_{B_{R/4}}
u^p(y)dy)^{\frac{p}{\gamma-1}}.
$$
By $p=\frac{n(\gamma-1)}{n-\beta\gamma}$, it follows
$$
\int_{A_R} u^p(x)dx \geq c(\int_{B_{R/4}}
u^p(y)dy)^{{\frac{p}{\gamma-1}}},
$$
where $c$ is independent of $R$. Letting $R \to \infty$, and noting
(\ref{tianim}), we obtain
$$
\int_{R^n} u^p(y)dy=0,
$$
which implies $u \equiv 0$. It is impossible.

The proof is complete.
\end{proof}

\paragraph{Remark 2.1.}
When $\beta=\alpha/2$ and $\gamma=2$, (\ref{wolff}) is reduced to
(\ref{HLS}). Theorem \ref{th2.5} is the generalization of Theorem
\ref{th2.3}.

\subsection{$\gamma$-Laplace equation}

\begin{theorem} \label{th2.6}
(1) If $p>\frac{n(\gamma-1)}{n-\gamma}$,
then the $\gamma$-Laplace equation
\begin{equation}
-\Delta_\gamma u=c(x)u^p
\label{pLaplace}
\end{equation}
has positive solutions for some double bounded $c(x)$.

(2) If $0<p \leq \frac{n(\gamma-1)}{n-\gamma}$, then for any double
bounded function $c(x)$, (\ref{pLaplace})
has no any positive solution satisfying $\inf_{R^n}u=0$.
\end{theorem}

\begin{proof}
(1) For $u(x)=\frac{1}{(1+|x|^m)^\theta}$ with $m=\frac{\gamma}{\gamma-1}$,
similar to the derivation of (\ref{2.0}), we have
\begin{equation} \label{2.3}
\begin{array}{ll}
&-\Delta_\gamma u=(1-\gamma)|U_r|^{\gamma-2}U_{rr}
-\frac{n-1}{r}|U_r|^{\gamma-2}U_r\\[3mm]
&=\frac{(m\theta)^{\gamma-1}r^{(m-1)(\gamma-2)+m-2}}
{(1+r^m)^{(\theta+1)(\gamma-1)}}
[\frac{-m(\theta+1)(\gamma-1)r^m}{1+r^m}+n-1+(\gamma-1)(m-1)]\\[3mm]
&=\frac{(m\theta)^{\gamma-1}}
{(1+r^m)^{(\theta+1)(\gamma-1)}}
[\frac{n+[n-(\theta+1)\gamma]r^m}{1+r^m}].
\end{array}
\end{equation}

Let $p>\frac{n(\gamma-1)}{n-\gamma}$. Take $m\theta=\frac{\gamma}{p-\gamma+1}$,
then $n>(\theta+1)\gamma$. Therefore, (\ref{2.3}) implies
$$
-\Delta_\gamma u=c(r)u^{(\theta+1)(\gamma-1)/\theta}=c(r)u^p
$$
for some double bounded $c(r)$. This result shows that (\ref{pLaplace}) has a
slowly decaying radial solution.

Moreover, if $p=\frac{n(\gamma-1)+\gamma}{n-\gamma}$,
we can find another fast decaying solution with rate
$m\theta=\frac{n-\gamma}{\gamma-1}$. Now, $n=(\theta+1)\gamma$ and hence
(\ref{pLaplace}) implies
$$
-\Delta_\gamma u=c(r)u^{[(\theta+1)(\gamma-1)+1]/\theta}
=c(r)u^p
$$
for some double bounded $c(r)$.

(2) Suppose $u$ solves (\ref{pLaplace}) and satisfies $\inf_{R^n}u=0$.
According to Corollary 4.13 in \cite{KM}, there exists $C>0$ such that
$$
\frac{1}{C}W_{1,\gamma}(cu^p)(x) \leq u(x) \leq CW_{1,\gamma}(cu^p)(x).
$$
Since $c(x)$ is double bounded, we can see that
$$
K(x):=\frac{u(x)}{W_{1,\gamma}(u^p)(x)}
$$
is also double bounded. This shows that $u$ solves
$$
u(x)=K(x)W_{1,\gamma}(u^p)(x).
$$
When $0<p \leq \frac{n(\gamma-1)}{n-\gamma}$, Theorem \ref{th2.5}
shows that this Wolff type equation has no positive solution for any
double bounded function $K(x)$.
Therefore, we prove the nonexistence of positive solutions to (\ref{pLaplace})
when $0<p\leq \frac{n(\gamma-1)}{n-\gamma}$.
\end{proof}

\section{Systems with variable coefficients}

\subsection{HLS type system}

\begin{theorem} \label{th3.1}
There exist positive solutions $u,v$ of
the integral system involving the Riesz potentials
\begin{equation}
 \left \{
   \begin{array}{l}
      u(x)=c_1(x) \displaystyle\int_{R^n}\frac{v^q(y)dy}{|x-y|^{n-\alpha}}\\
      v(x)=c_2(x)\displaystyle\int_{R^n}\frac{u^p(y)dy}{|x-y|^{n-\alpha}}
   \end{array}
   \right. \label{HLSs}         
 \end{equation}
for some double bounded functions $c_1(x)$ and $c_2(x)$,
if and only if $pq>1$ and
\begin{equation} \label{5.1}
\max\{\frac{\alpha(p+1)}{pq-1}, \frac{\alpha(q+1)}{pq-1}\} < n-\alpha.
\end{equation}
\end{theorem}

\begin{proof}

{\it Step 1.} Sufficiency.

Set
\begin{equation} \label{5.2}
u(x)=\frac{1}{(1+|x|^2)^{\theta_1}}, \quad
v(x)=\frac{1}{(1+|x|^2)^{\theta_2}}.
\end{equation}
Similar to the argument in the proof of Theorem \ref{th2.3}, we can
find four pairs solutions.

(i) Take the slow rates
$$
2\theta_1=\frac{\alpha(q+1)}{pq-1}, \quad
2\theta_2=\frac{\alpha(p+1)}{pq-1}.
$$
Then $pq>1$ as well as (\ref{5.1}) lead to
$\alpha<2p\theta_1< n$ and $\alpha<2q\theta_2< n$.
Therefore,
$$
\int_{R^n}\frac{u^p(y)dy}{|x-y|^{n-\alpha}}=\frac{c_1(x)}{(1+|x|^2)^{p\theta_1-\alpha/2}}=c_1(x)v(x),
$$
$$
\int_{R^n}\frac{v^q(y)dy}{|x-y|^{n-\alpha}}=\frac{c_2(x)}{(1+|x|^2)^{q\theta_2-\alpha/2}}=c_2(x)u(x),
$$
for some double bounded functions $c_1(x)$ and $c_2(x)$.
This consequence shows that (\ref{HLSs}) has a pair of radial solutions $(u,v)$ as (\ref{5.2}).

(ii) Moreover, if the stronger condition $p,q> \frac{n}{n-\alpha}$
holds, then we can find solutions $u,v$ with
the fast decay rate $2\theta_1=2\theta_2=n-\alpha$.
Now, $2p\theta_1> n$ and $2q\theta_2> n$,
then
$$
\int_{R^n}\frac{u^p(y)dy}{|x-y|^{n-\alpha}}=\frac{c_1(x)}{(1+|x|^2)^{(n-\alpha)/2}}=c_1(x)v(x),
$$
$$
\int_{R^n}\frac{v^q(y)dy}{|x-y|^{n-\alpha}}=\frac{c_2(x)}{(1+|x|^2)^{(n-\alpha)/2}}=c_2(x)u(x),
$$
for some double bounded functions $c_1(x),c_2(x)$. Therefore,
(\ref{HLSs}) has a pair of radial solutions $(u,v)$ as (\ref{5.2}).

(iii) If another stronger condition $pq>1$ as well as
\begin{equation} \label{cn1}
\frac{\alpha}{n-\alpha}<p < \frac{n}{n-\alpha},
\end{equation}
\begin{equation} \label{cn2}
\frac{(q+1)\alpha}{pq-1} < n-\alpha,
\end{equation}
holds, we can find a pair of solutions $u,v$. Now, $u,v$ have two
different fast decay rates.

We claim that if $pq>1$,
the condition (\ref{cn1}) together with (\ref{cn2}) are stronger than (\ref{5.1}).
In fact, we first see $p \leq q$. Otherwise,
(\ref{cn1}) implies $q<p< \frac{n}{n-\alpha}$, which means $q[p(n-\alpha)-\alpha]<n$.
This contradicts (\ref{cn2}).
From (\ref{cn2}) and $p \leq q$, it follows that $pq(n-\alpha)-n > q\alpha \geq p\alpha$. This leads to
$\frac{(p+1)\alpha}{pq-1} < n-\alpha$. Combining this with (\ref{cn2}) yields (\ref{5.1}).

Take $2\theta_1=n-\alpha$,
$2\theta_2=2p\theta_1-\alpha=p(n-\alpha)-\alpha$.
Then (\ref{cn1}) and (\ref{cn2}) lead to
$\alpha<2p\theta_1< n$ and $2q\theta_2> n$.
Therefore,
$$
\int_{R^n}\frac{u^p(y)dy}{|x-y|^{n-\alpha}}=\frac{c_1(x)}{(1+|x|^2)^{p\theta_1-\alpha/2}}=c_1(x)v(x),
$$
$$
\int_{R^n}\frac{v^q(y)dy}{|x-y|^{n-\alpha}}=\frac{c_2(x)}{(1+|x|^2)^{(n-\alpha)/2}}=c_2(x)u(x),
$$
for another double bounded functions $c_1(x),c_2(x)$. Therefore,
(\ref{HLSs}) has a pair of radial solutions $(u,v)$ as (\ref{5.2}).

By the same argument above, we know that once $pq>1$ as well as
$\frac{\alpha}{n-\alpha}<q < \frac{n}{n-\alpha}$ and
$\frac{(p+1)\alpha}{pq-1} < n-\alpha$, (\ref{HLSs}) has a pair of
radial solutions $(u,v)$ as (\ref{5.2}). Now, $u,v$ decay fast by
two different rates.

(iv) We can find another pair of radial solutions to (\ref{HLSs}).
They decay with fast rates which are different from (\ref{5.2}).
Now, we assume
$$
u(x)=\frac{1}{(1+|x|^2)^{(n-\alpha)/2}}; \quad
v(x)=\frac{\log|x|}{(1+|x|^2)^{(n-\alpha)/2}}.
$$
It is easy to verify that $u,v$ also solve (\ref{HLSs}) with some
double bounded functions $c_1,c_2$.

{\it Note.} According to Corollary 1.3 (2) in \cite{SL}, if $(u,v)
\in L^{p+1}(R^n) \times L^{q+1}(R^n)$ where $p,q$ satisfy the
critical condition
$\frac{1}{p+1}+\frac{1}{q+1}=1-\frac{\alpha}{n}$, then $u,v$ decay
with only three rates as in (ii)-(iv).

{\it Step 2.} Necessity.

(i) If either $0<pq \leq 1$ or $pq>1$ and
$\max\{\frac{(p+1)\alpha}{pq-1}, \frac{(q+1)\alpha}{pq-1}\} >
n-\alpha$, we prove the nonexistence.

Assume $u,v$ are positive solutions of (\ref{HLSs}). First, for $|x|>R$,
$$
u(x) \geq c\int_{B_R(0)} \frac{dy}{|x-y|^{n-\alpha}} \geq \frac{c}{|x|^{a_0}}.
$$
Here $a_0=n-\alpha$. By this estimate, for $|x|>R$ there holds
$$
v(x) \geq c\int_{B_{|x|/2}(x)}\frac{|y|^{-pa_0}dy}{|x-y|^{n-\alpha}} \geq \frac{c}{|x|^{b_0}},
$$
where $b_0=pa_0-\alpha$. This implies
$$
u(x) \geq
c\int_{B_{|x|/2}(x)}\frac{|y|^{-qb_0}dy}{|x-y|^{n-\alpha}} \geq
\frac{c}{|x|^{a_1}},
$$
for $|x|>R$, where $a_1=qb_0-\alpha$. By induction, we obtain
that for $|x|>R$,
$$
v(x) \geq \frac{c}{|x|^{b_k}}, \quad u(x) \geq \frac{c}{|x|^{a_k}}.
$$
Here $a_0=n-\alpha$, $b_k=pa_k-\alpha$ and
$a_k=qb_{k-1}-\alpha$. Therefore, we have
$$\begin{array}{ll}
a_j&=pqa_{j-1}-\alpha(q+1)=(pq)^2a_{j-2}-\alpha(q+1)(1+pq)
=\cdots\\[3mm]
&=(pq)^ja_0-\alpha(q+1)[1+pq+\cdots+(pq)^{j-1}].
\end{array}
$$

Case I:
When $pq=1$, for some large $j_0$, it follows
$$
a_{j_0}=a_0-\alpha(q+1)j_0<0.
$$

Case II:
When $0<pq<1$, letting $j \to \infty$, we get
$$
a_j=(pq)^{j}a_0-\alpha(q+1)\frac{1-(pq)^{j}}{1-pq} \to -\frac{\alpha(q+1)}{1-pq}<0.
$$
Therefore, we can find $j_0$ such that $a_{j_0}<0$.

Case III: When $pq>1$ and $\frac{\alpha(q+1)}{pq-1} > n-\alpha$.

Now, $a_0<\frac{\alpha (q+1)}{pq-1}$.
Thus, we deduce that for some large $j_0$,
$$
a_{j_0}=(pq)^{j_0}a_0-\alpha(q+1)\frac{(pq)^{j_0}-1}{pq-1}
=(pq)^{j_0}[a_0-\frac{\alpha (q+1)}{pq-1}]+\frac{\alpha (q+1)}{pq-1}<0.
$$

Case IV: When $pq>1$ and $\frac{\alpha(p+1)}{pq-1} > n-\alpha$.

By an analogous argument of Case III, we can also find some $k_0$ such that $b_{k_0}<0$.

These results imply $u(x)=\infty$ or $v(x)=\infty$. It is
impossible. The contradiction shows the nonexistence of the
positive solutions to (\ref{HLSs}).

(ii) If $pq>1$ and $\max\{\frac{(p+1)\alpha}{pq-1}, \frac{(q+1)\alpha}{pq-1}\}
=n-\alpha$, we prove the nonexistence.

The idea is the same as Step 3 in the proof of Theorem \ref{2.3}.
Denote $B_R(0)$ by $B$. First,
$$
u(x) \geq \frac{c}{(R+|x|)^{n-\alpha}}\int_B v^q(y)dy, \quad
v(x) \geq \frac{c}{(R+|x|)^{n-\alpha}}\int_B u^p(y)dy.
$$
Thus,
$$
\int_B u^p(x)dx \geq \frac{c}{R^{p(n-\alpha)-n}}(\int_B v^q(y)dy)^p,
$$
$$
\int_B v^q(x)dx \geq \frac{c}{R^{q(n-\alpha)-n}}(\int_B u^p(y)dy)^q.
$$
Without loss of generality, assume $p \leq q$.
Combining two results above with $\frac{(q+1)\alpha}{pq-1}
=n-\alpha$ yields
$$
\int_B v^q(x)dx \geq c(\int_B v^q(y)dy)^{pq},
$$
where $c$ is independent of $R$. Letting $R \to \infty$, we get
$v \in L^q(R^n)$. On the other hand, we also obtain
$$
\int_{A_R} v^q(x)dx \geq c(\int_B v^q(y)dy)^{pq}.
$$
Letting $R \to \infty$ and noting $v \in L^q(R^n)$,
we see $v \equiv 0$. It is impossible.

Theorem \ref{th3.1} is proved.
\end{proof}

\begin{corollary} \label{coro3.2}
Let $k \in [1,n/2)$ be an integer and $pq>1$. There exist positive solutions $u,v$ of
the semilinear Lane-Emden type system
\begin{equation}
 \left \{
   \begin{array}{l}
      (-\Delta)^k u(x)=c_1(x)v^q(x)\\
      (-\Delta)^k v(x)=c_2(x)u^p(x)
   \end{array}
   \right. \label{LEs}         
 \end{equation}
for some double bounded functions $c_1(x)$ and $c_2(x)$,
if and only if
$$
\max\{\frac{2k(p+1)}{pq-1},
\frac{2k(q+1)}{pq-1}\} < n-2k.
$$
\end{corollary}

\begin{proof}
When $pq>1$, Liu, Guo and Zhang \cite{LGZ} proved $(-\Delta)^i u>0$ and $(-\Delta)^i v>0$.
Similar to the argument in \cite{CL-2010} we can also establish the
equivalence between (\ref{LEs}) and (\ref{HLSs}). So Corollary \ref{coro3.2}
is a direct corollary of Theorem \ref{th3.1} with $\alpha=2k$.
\end{proof}

\subsection{Wolff type system}

\begin{theorem} \label{th3.3}
There exist positive solutions $u,v$ of
the integral system involving the Wolff potentials
\begin{equation}
 \left \{
   \begin{array}{l}
      u(x)=c_1(x)W_{\beta,\gamma}(v^q)(x)\\
      v(x)=c_2(x)W_{\beta,\gamma}(u^p)(x)
   \end{array}
   \right. \label{Wolffs}          
 \end{equation}
for some double bounded functions $c_1(x)$ and $c_2(x)$,
if and only if $pq>(\gamma-1)^2$ and
\begin{equation} \label{5.5}
\max\{\frac{\beta\gamma(p+\gamma-1)}{pq-(\gamma-1)^2},
\frac{\beta\gamma(q+\gamma-1)}{pq-(\gamma-1)^2}\} < \frac{n-\beta\gamma}{\gamma-1}.
\end{equation}.
\end{theorem}

\begin{proof}
{\it Step 1.} Existence.

Insert (\ref{5.2}) into $W_{\beta,\gamma}(u^p)$ and $W_{\beta,\gamma}(v^q)$.
Similar to the argument in the proof of Theorem \ref{th2.5}, we also discuss
in four cases.

(i) Take the slow rates
$$
2\theta_1=\frac{\beta\gamma(q+\gamma-1)}{pq-(\gamma-1)^2}, \quad
2\theta_2=\frac{\beta\gamma(p+\gamma-1)}{pq-(\gamma-1)^2}.
$$
Then, $pq>(\gamma-1)^2$ and (\ref{5.5}) lead to
$\beta\gamma<2p\theta_1< n$ and $\beta\gamma<2q\theta_2< n$.
Therefore,
$$
W_{\beta,\gamma}(u^p)(x)=\frac{c_1(x)}{(1+|x|^2)^{\frac{2p\theta_1-\beta\gamma}
{2(\gamma-1)}}}=c_1(x)v(x),
$$
$$
W_{\beta,\gamma}(v^q)(x)=\frac{c_2(x)}{(1+|x|^2)^{\frac{2q\theta_2-\beta\gamma}
{2(\gamma-1)}}}=c_2(x)u(x),
$$
for some double bounded functions $c_1(x)$, $c_2(x)$.
This implies that
(\ref{Wolffs}) has a pair of radial solutions $(u,v)$ as (\ref{5.2}).

(ii) If the stronger condition $p,q> \frac{n(\gamma-1)}{n-\beta\gamma}$
holds, we take the fast rate $2\theta_1=2\theta_2
=\frac{n-\beta\gamma}{\gamma-1}$.
Then $2p\theta_1> n$ and $2q\theta_2> n$,
and hence
$$
W_{\beta,\gamma}(u^p)(x)=\frac{c_1(x)}{(1+|x|^2)^{\frac{n-\beta\gamma}{2(\gamma-1)}}}
=c_1(x)v(x),
$$
$$
W_{\beta,\gamma}(v^q)(x)=\frac{c_2(x)}{(1+|x|^2)^{\frac{n-\beta\gamma}{2(\gamma-1)}}}
=c_2(x)u(x),
$$
for another double bounded functions $c_1(x), c_2(x)$. This implies that
(\ref{Wolffs}) has a pair of radial solutions $(u,v)$ as (\ref{5.2})
with fast decay rates.

(iii) Similar to the argument in Theorem \ref{th3.1},
if $pq>(\gamma-1)^2$, the condition
$$
\frac{\beta\gamma}{n-\beta\gamma}<p< \frac{n(\gamma-1)}{n-\beta\gamma}, \quad
\frac{\beta\gamma(q+\gamma-1)}{pq-(\gamma-1)^2} < \frac{n-\beta\gamma}{\gamma-1}
$$
is also stronger than (\ref{5.5}). When this stronger condition
holds, then we take $2\theta_1=\frac{n-\beta\gamma}{\gamma-1}$, $2\theta_2=
\frac{2p\theta_1-\beta\gamma}{\gamma-1}=p\frac{n-\beta\gamma}{(\gamma-1)^2}
-\frac{\beta\gamma}{\gamma-1}$. Therefore,
$\beta\gamma<2p\theta_1< n$ and $2q\theta_2> n$,
and hence
$$
W_{\beta,\gamma}(u^p)(x)=\frac{c_1(x)}{(1+|x|^2)^{\frac{2p\theta_1-\beta\gamma}{2(\gamma-1)}}}
=c_1(x)v(x),
$$
$$
W_{\beta,\gamma}(v^q)(x)=\frac{c_2(x)}{(1+|x|^2)^{\frac{n-\beta\gamma}{2(\gamma-1)}}}
=c_2(x)u(x),
$$
for another double bounded functions $c_1(x)$, $c_2(x)$. This shows
(\ref{Wolffs}) has radial solutions as (\ref{5.2}).

Similar to the argument above, if another stronger condition
$pq>(\gamma-1)^2$ as well as
$$
\frac{\beta\gamma}{n-\beta\gamma}<q< \frac{n(\gamma-1)}{n-\beta\gamma},\quad
\frac{\beta\gamma(p+\gamma-1)}{pq-(\gamma-1)^2} < \frac{n-\beta\gamma}{\gamma-1}
$$
holds, (\ref{Wolffs}) also has radial solutions as (\ref{5.2})
with two different fast rates
$2\theta_2=\frac{n-\beta\gamma}{\gamma-1}$, $2\theta_1=
q\frac{n-\beta\gamma}{(\gamma-1)^2}-\frac{\beta\gamma}{\gamma-1}$.

(iv) Eq. (\ref{Wolffs}) also has another pair of radial solutions
which also decay fast by two different rates. One decays with
$\frac{n-\beta\gamma}{\gamma-1}$, and another decays with
logarithmic order. Now, we assume
$$
u(x)=\frac{1}{(1+|x|^2)^{\frac{n-\beta\gamma}{2(\gamma-1)}}};
\quad v(x)=\frac{(\log|x|)^{\frac{1}{\gamma-1}}}
{(1+|x|^2)^{\frac{n-\beta\gamma}{2(\gamma-1)}}}.
$$
It is easy to verify that $u,v$ solve (\ref{Wolffs}) with some
double bounded functions $c_1,c_2$.

{\it Step 2.} Nonexistence.

{\it Substep 2.1.} Suppose either $0<pq \leq (\gamma-1)^2$ or
$$
\max\{\frac{\beta\gamma(p+\gamma-1)}{pq-(\gamma-1)^2},
\frac{\beta\gamma(q+\gamma-1)}{pq-(\gamma-1)^2}\} > \frac{n-\beta\gamma}{\gamma-1}.
$$

Assume $u,v$ are positive solutions of (\ref{Wolffs}).
Noting $\int_{B_R(0)}v^q(y)dy \geq c$, we obtain that for $|x|>R$,
$$
u(x) \geq \int_{|x|+R}^\infty(\frac{\int_{B_R(0)} v^q(y)dy}
{t^{n-\beta\gamma}})^{\frac{1}{\gamma-1}}  \frac{dt}{t}
\geq c\int_{|x|+R}^\infty t^{-\frac{n-\beta\gamma}{\gamma-1}}\frac{dt}{t}
\geq \frac{c}{|x|^{a_0}}.
$$
Here $a_0=\frac{n-\beta\gamma}{\gamma-1}$. By this estimate, for $|x|>R$, there holds
$$
v(x) \geq c\int_{2|x|}^\infty[\int_{B_{t-|x|}(0) \setminus B_{(t-|x|)/2}(0)}
\frac{dy}{|y|^{pa_0}} t^{\beta\gamma-n}]^{\frac{1}{\gamma-1}} \frac{dt}{t}
\geq c\int_{2|x|}^\infty t^{\frac{\beta\gamma-pa_0}{\gamma-1}} \frac{dt}{t}.
$$

When $\beta\gamma-pa_0 \geq 0$, we see $v(x)=\infty$ for $|x|>R$.
This implies the nonexistence
of positive solutions of (\ref{Wolffs}) since $R$ is an arbitrary positive number.
When $\beta\gamma-pa_0 < 0$, then
$$
v(x) \geq \frac{c}{|x|^{b_0}}, \quad for ~|x|>R,
$$
where $b_0=\frac{pa_0-\beta\gamma}{\gamma-1}$. Similarly, using this estimate, we also obtain
that if $\beta\gamma-qb_0 \geq 0$, then $u(x)=\infty$; if $\beta\gamma-qb_0<0$, then
$$
u(x) \geq \frac{c}{|x|^{a_1}}, \quad for ~|x|>R,
$$
where $a_1=qb_0-\beta\gamma$.

For $k=1,2,\cdots$, write
$$
a_0=\frac{n-\beta\gamma}{\gamma-1}, \quad
b_k=\frac{pa_k-\beta\gamma}{\gamma-1}, \quad
a_k=\frac{qb_{k-1}-\beta\gamma}{\gamma-1}.
$$

By induction, we can obtain the following conclusions:

(i) If $a_k < 0$, then $u(x)=\infty$. This leads to the nonexistence.
If $a_k \geq 0$, then $u(x) \geq \frac{c}{|x|^{a_k}}$ implies
$v(x) \geq \frac{c}{|x|^{b_k}}$.

(ii) If $b_k < 0$, then $v(x)=\infty$. This also leads to the nonexistence.
If $b_k \geq 0$, then $v(x) \geq \frac{c}{|x|^{b_k}}$ implies
$u(x) \geq \frac{c}{|x|^{a_{k+1}}}$.

In view of
$$
a_k=\frac{q}{\gamma-1}b_{k-1}-\frac{\beta\gamma}{\gamma-1}
=\frac{pq}{(\gamma-1)^2}a_{k-1}
-\frac{\beta\gamma}{\gamma-1}\frac{q+\gamma-1}{\gamma-1},
$$
we deduce that
$$\begin{array}{ll}
a_j&=\displaystyle\frac{pq}{(\gamma-1)^2}a_{j-1}-\frac{\beta\gamma}{\gamma-1}
\frac{q+\gamma-1}{\gamma-1}\\[3mm]
&=(\displaystyle\frac{pq}{(\gamma-1)^2})^2a_{j-2}-\frac{\beta\gamma}{\gamma-1}
\frac{q+\gamma-1}{\gamma-1}(1+\frac{pq}{(\gamma-1)^2})
=\cdots\\[3mm]
&=(\displaystyle\frac{pq}{(\gamma-1)^2})^ja_0-\frac{\beta\gamma}{\gamma-1}\frac{q+\gamma-1}{\gamma-1}
[1+\frac{pq}{(\gamma-1)^2}+\cdots+(\frac{pq}{(\gamma-1)^2})^{j-1}].
\end{array}
$$

When $\frac{pq}{(\gamma-1)^2}=1$, then for some large $j_0$,
$$
a_{j_0}=a_0-\frac{\beta\gamma}{\gamma-1}\frac{q+\gamma-1}{\gamma-1}j_0
< 0.
$$
This implies $u(x)=\infty$.

When $0<\frac{pq}{(\gamma-1)^2}<1$, letting $j \to \infty$, we get
$$
a_j=(\frac{pq}{(\gamma-1)^2})^{j}a_0-\frac{\beta\gamma}{\gamma-1}\frac{q+\gamma-1}{\gamma-1}
\frac{1-(\frac{pq}{(\gamma-1)^2})^{j}}{1-\frac{pq}{(\gamma-1)^2}}
\to -\frac{\beta\gamma(q+\gamma-1)}{(\gamma-1)^2-pq}<0.
$$
Therefore, we can find $j_0$ such that $a_{j_0}< 0$.
This implies $u(x)=\infty$.

When $\frac{pq}{(\gamma-1)^2}>1$ and
$\frac{\beta\gamma(q+\gamma-1)}{pq-(\gamma-1)^2} > \frac{n-\beta\gamma}{\gamma-1}$,
there holds $a_0<\frac{\beta\gamma(q+\gamma-1)}{pq-(\gamma-1)^2}$.
We deduce that
$$\begin{array}{ll}
a_{j_0}&=(\displaystyle\frac{pq}{(\gamma-1)^2})^{j_0}a_0
-\frac{\beta\gamma}{\gamma-1}\frac{q+\gamma-1}{\gamma-1}
\frac{(\frac{pq}{(\gamma-1)^2})^{j_0}-1}{\frac{pq}{(\gamma-1)^2}-1}\\[3mm]
&=(\displaystyle\frac{pq}{(\gamma-1)^2})^{j_0}
[a_0-\frac{\beta\gamma(q+\gamma-1)}{pq-(\gamma-1)^2}]
+\frac{\beta\gamma(q+\gamma-1)}{pq-(\gamma-1)^2}<0
\end{array}
$$
for some large $j_0$. We also see $u(x)=\infty$.

When $\frac{pq}{(\gamma-1)^2}>1$ and
$\frac{\beta\gamma(p+\gamma-1)}{pq-(\gamma-1)^2} >\frac{n-\beta\gamma}{\gamma-1}$,
there also holds
$a_0<\frac{\beta\gamma(p+\gamma-1)}{pq-(\gamma-1)^2}$.
By the same argument above, we
handle $b_k$ instead of $a_k$, we can also find some $k_0$
such that $b_{k_0}<0$. This implies
$v(x)=\infty$.

{\it Substep 2.2.} Suppose $pq>(\gamma-1)^2$ and
$$
\max\{\frac{\beta\gamma(p+\gamma-1)}{pq-(\gamma-1)^2},
\frac{\beta\gamma(q+\gamma-1)}{pq-(\gamma-1)^2}\} = \frac{n-\beta\gamma}{\gamma-1}.
$$

First, write $H:=\int_{B_t(x)}v^q(y)dy$. By the H\"older inequality,
$$\begin{array}{ll}
\displaystyle\int_0^R H dt &\leq (\displaystyle\int_0^R H^{\frac{1}{\gamma-1}}
t^{\frac{\beta\gamma-n}{\gamma-1}-1}dt)^{\gamma-1}
(\int_0^R t^{\frac{n-\beta\gamma+\gamma-1}{2-\gamma}}dt)^{2-\gamma}\\[3mm]
&=CR^{n-\beta\gamma+1} (\displaystyle\int_0^R (\frac{H}{t^{n-\beta\gamma}})^{\frac{1}{\gamma-1}}
\frac{dt}{t})^{\gamma-1}.
\end{array}
$$
Therefore, exchanging the order of variables yields
$$\begin{array}{ll}
u(x) \geq c\displaystyle\int_0^R (\frac{H}{t^{n-\beta\gamma}})^{\frac{1}{\gamma-1}}
\frac{dt}{t} &\geq cR^{-\frac{n-\beta\gamma+1}{\gamma-1}}(\displaystyle\int_0^RHdt)^{\frac{1}{\gamma-1}}\\[3mm]
&\geq cR^{-\frac{n-\beta\gamma}{\gamma-1}}(\displaystyle\int_{B_{R/4}}v^q(y)dy)^{\frac{1}{\gamma-1}}.
\end{array}
$$
Thus,
\begin{equation} \label{santian}
u^p(x) \geq cR^{-p\frac{n-\beta\gamma}{\gamma-1}}(\int_{B_{R/4}}v^q(y)dy)^{\frac{p}{\gamma-1}}.
\end{equation}
Similarly,
\begin{equation} \label{ertian}
v^q(x) \geq cR^{-q\frac{n-\beta\gamma}{\gamma-1}}(\int_{B_{R/4}}u^p(y)dy)^{\frac{q}{\gamma-1}}.
\end{equation}

Without loss of generality, we suppose
\begin{equation} \label{yitian}
\frac{\beta\gamma(q+\gamma-1)}{pq-(\gamma-1)^2}=\frac{n-\beta\gamma}{\gamma-1}.
\end{equation}
Inserting (\ref{santian}) into (\ref{ertian}) yields
\begin{equation} \label{wutian}
v^q(x) \geq cR^{-q\frac{n-\beta\gamma}{\gamma-1}-pq\frac{n-\beta\gamma}{(\gamma-1)^2}+\frac{nq}{\gamma-1}}
(\int_{B_{R/4}}v^q(y)dy)^{\frac{pq}{(\gamma-1)^2}}.
\end{equation}
Integrating on $B_{R/4}$, we get
\begin{equation} \label{sitian}
\int_{B_{R/4}}v^q(x)dx \geq cR^{-q\frac{n-\beta\gamma}{\gamma-1}(1+\frac{p}{\gamma-1})+n(\frac{q}{\gamma-1}+1)}
(\int_{B_{R/4}}v^q(y)dy)^{\frac{pq}{(\gamma-1)^2}}.
\end{equation}

We claim that the exponent of $R$ is zero. In fact, $q\beta\gamma+n(\gamma-1)
=\beta\gamma(q+\gamma-1)+(n-\beta\gamma)(\gamma-1)$. By (\ref{yitian}), we obtain
$$
q\beta\gamma+n(\gamma-1)=[pq-(\gamma-1)^2]\frac{n-\beta\gamma}{\gamma-1}+
(\gamma-1)^2\frac{n-\beta\gamma}{\gamma-1}=pq\frac{n-\beta\gamma}{\gamma-1}.
$$
Multiplying by $(\gamma-1)^{-1}$, we have
$$
n(\frac{q}{\gamma-1}+1)=q\frac{n-\beta\gamma}{\gamma-1}+\frac{pq}{\gamma-1}
\frac{n-\beta\gamma}{\gamma-1}=q\frac{n-\beta\gamma}{\gamma-1}(1+\frac{p}{\gamma-1}).
$$
The claim is proved.

Letting $R \to \infty$ in (\ref{sitian}), we see that $v \in L^q(R^n)$ in view of $pq>(\gamma-1)^2$.

Integrating (\ref{wutian}) on $A_R:=B_{R/4} \setminus B_{R/8}$ and letting $R \to \infty$, we also have
$\int_{R^n} v^q(y) dy =0$. It is impossible.

Thus, we complete our proof.
\end{proof}

\subsection{$\gamma$-Laplace system}

\begin{theorem} \label{th3.4}
(1) If $pq>(\gamma-1)^2$ and
\begin{equation} \label{tiao}
\max\{\frac{\gamma(q+\gamma-1)}{pq-(\gamma-1)^2},
\frac{\gamma(p+\gamma-1)}{pq-(\gamma-1)^2}\}<\frac{n-\gamma}{\gamma-1},
\end{equation}
then there exist positive solutions $u,v$ of the
$\gamma$-Laplace system
\begin{equation}
 \left \{
   \begin{array}{l}
      -\Delta_\gamma u(x)=c_1(x)v^q(x), \quad x \in R^n,\\
      -\Delta_\gamma v(x)=c_2(x)u^p(x), \quad x \in R^n
   \end{array}
   \right.   \label{pLs}       
 \end{equation}
for some double bounded $c_1(x)$ and $c_2(x)$.

(2) For any double bounded functions $c_1(x)$ and $c_2(x)$, if one of the
following conditions holds:

(i) $0<pq\leq (\gamma-1)^2$;

(ii) $pq>(\gamma-1)^2$ and
\begin{equation} \label{you1}
\max\{\frac{\gamma(q+\gamma-1)}{pq-(\gamma-1)^2},
\frac{\gamma(p+\gamma-1)}{pq-(\gamma-1)^2}\} \geq \frac{n-\gamma}{\gamma-1},
\end{equation}
then (\ref{pLs}) has no positive solutions $u,v$ satisfying
$\inf_{R^n}u=\inf_{R^n}v=0$.
\end{theorem}

\begin{proof}

(1) {\it Existence}.

Let $m=\frac{\gamma}{\gamma-1}$. Take
$$
u(x)=\frac{1}{(1+|x|^m)^{\theta_1}},\quad
v(x)=\frac{1}{(1+|x|^m)^{\theta_2}}.
$$
Similar to the calculation in (\ref{2.3}),
we also obtain
$$\begin{array}{ll}
-\Delta_\gamma u(x)&=\frac{(m\theta_1)^{\gamma-1}}
{(1+r^m)^{(\theta_1+1)(\gamma-1)}}
[\frac{n+[n-(\theta_1+1)\gamma]r^m}{1+r^m}],
\end{array}
$$
$$\begin{array}{ll}
-\Delta_\gamma v(x)&=\frac{(m\theta_2)^{\gamma-1}}
{(1+r^m)^{(\theta_2+1)(\gamma-1)}}
[\frac{n+[n-(\theta_2+1)\gamma]r^m}{1+r^m}].
\end{array}
$$
Therefore, the signs of both sides of the results above show four
cases.

(i) Take the slow decay rates
$$
m\theta_1=\frac{\gamma(q+\gamma-1)}{pq-(\gamma-1)^2}, \quad
m\theta_2=\frac{\gamma(p+\gamma-1)}{pq-(\gamma-1)^2}.
$$
Then $pq>(\gamma-1)^2$ and (\ref{tiao}) lead to
\begin{equation} \label{jian}
(\theta_1+1)\gamma<n, \quad \hbox{and}
\quad (\theta_2+1)\gamma<n,
\end{equation}
and hence
$$
-\Delta_\gamma u(x)=\frac{c_1(r)}{(1+r^m)^{(\theta_1+1)(\gamma-1)}}
=c_1(x)v^q(x),
$$
$$
-\Delta_\gamma v(x)=\frac{c_2(r)}{(1+r^m)^{(\theta_2+1)(\gamma-1)}}
=c_2(x)u^p(x).
$$
This shows that
(\ref{pLs}) has the radial solutions as (\ref{5.2})
with slow decay rates.

(ii) Moreover, if $p=q=\frac{n(\gamma-1)+\gamma}{n-\gamma}$,
then we take the fast decay rates $m\theta_1=m\theta_2=\frac{n-\gamma}{\gamma-1}$.
This leads to $n=(\theta_1+1)\gamma=(\theta_2+1)\gamma$. Therefore,
$$
-\Delta_\gamma u(x)=\frac{c_1(r)}{(1+r^m)^{(\theta_1+1)(\gamma-1)+1}}
=c_1(x)v^q(x),
$$
$$
-\Delta_\gamma v(x)=\frac{c_2(r)}{(1+r^m)^{(\theta_2+1)(\gamma-1)+1}}
=c_2(x)u^p(x).
$$
This shows that
(\ref{pLs}) has the radial solutions as (\ref{5.2})
with fast decay rates.

(iii) If $\frac{\gamma(q+\gamma)}{pq-(\gamma-1)^2}=\frac{n-\gamma}{\gamma-1}$,
then we take other fast decay rates $m\theta_1=\frac{n-\gamma}{\gamma-1}$,
$m\theta_2=p\frac{n-\gamma}{(\gamma-1)^2}-\frac{\gamma}{\gamma-1}$. Thus,
$n=(\theta_1+1)\gamma$, $n>(\theta_2+1)\gamma$. Therefore,
$$
-\Delta_\gamma u(x)=\frac{c_1(r)}{(1+r^m)^{(\theta_1+1)(\gamma-1)+1}}
=c_1(x)v^q(x),
$$
$$
-\Delta_\gamma v(x)=\frac{c_2(r)}{(1+r^m)^{(\theta_2+1)(\gamma-1)}}
=c_2(x)u^p(x).
$$
This shows that
(\ref{pLs}) has the radial solutions as (\ref{5.2})
with the second fast decay rates.

Similar to the argument above, if $\frac{\gamma(p+\gamma)}{pq-(\gamma-1)^2}
=\frac{n-\gamma}{\gamma-1}$
holds, (\ref{Wolffs}) also has radial solutions as (\ref{5.2})
with the third fast rates
$m\theta_2=\frac{n-\gamma}{\gamma-1}$, $m\theta_1=
q\frac{n-\gamma}{(\gamma-1)^2}-\frac{\gamma}{\gamma-1}$.

(iv) Eq. (\ref{Wolffs}) also has another pair of radial solutions
which also decay fast with the different rates. One decays with
$\frac{n-\gamma}{\gamma-1}$, and another decays with
logarithmic order. Now, we assume
$$
u(x)=\frac{1}{(1+|x|^m)^{\frac{n-\gamma}{\gamma}}};
\quad v(x)=\frac{(\log|x|)^{\frac{1}{\gamma-1}}}
{(1+|x|^m)^{\frac{n-\gamma}{\gamma}}}.
$$
It is easy to verify that $u,v$ solve (\ref{Wolffs}) with some
double bounded functions $c_1,c_2$.

(2) {\it Nonexistence}.

Suppose $u,v$ are positive solutions of (\ref{pLs}) satisfying
$\inf_{R^n}u=\inf_{R^n}v=0$. According to Corollary 4.13 in \cite{KM},
there exists $C>0$ such that
$$
\frac{1}{C}W_{1,\gamma}(c_1v^q)(x) \leq u(x) \leq CW_{1,\gamma}(c_1v^q)(x),
$$
$$
\frac{1}{C}W_{1,\gamma}(c_2u^p)(x) \leq v(x) \leq CW_{1,\gamma}(c_2u^p)(x).
$$
Since $c_1$ and $c_2$ are double bounded, we can find two other double bounded functions
$K_1(x)$ and $K_2(x)$ such that
$$
u(x)=K_1(x)W_{1,\gamma}(v^q)(x), \quad v(x)=K_2(x)W_{1,\gamma}(u^p)(x).
$$
By Theorem \ref{th3.3} with $\beta=1$, we can see the nonexistence.
\end{proof}

\section{Finite energy solutions: scalar equations}

In this section, we consider the critical conditions associated with the
existence of the positive solutions when the coefficient $c(x) \equiv Constant$.
Without loss of generality, we take $c(x) \equiv 1$.

\subsection{Critical exponents and scaling invariants}

Take a scaling transform
$u_\mu(x)=\mu^{\frac{n-2}{2}} u(\mu x)$.
Assume $u$ solves $-\Delta u=u^{\frac{n+2}{n-2}}$.
By a simply calculation, we have
$$
-\Delta u_\mu=u_\mu^{\frac{n+2}{n-2}} \quad and
\quad \|u\|_{\frac{2n}{n-2}}=\|u_\mu\|_{\frac{2n}{n-2}}.
$$
For the higher order equation, the corresponding result above is still true.

Furthermore, we have the more general result.

\begin{theorem} \label{th4.1}
The HLS type eauation
\begin{equation}
u(x)=\int_{R^n}\frac{u^p(y)dy}{|x-y|^{n-\alpha}}
\label{3.1}
\end{equation}
and the energy $\|u\|_{L^{p+1}(R^n)}$
are invariant under the scaling transform, if and only if
\begin{equation}
p=\frac{n+\alpha}{n-\alpha}.
\label{3.2}
\end{equation}
\end{theorem}

\begin{proof}
Take the scaling transform
$$
u_\mu(x)=\mu^\sigma u(\mu x).
$$
Then
$$\begin{array}{ll}
u_\mu(x)&=\mu^\sigma\displaystyle\int_{R^n}\frac{u^p(y)dy}{|\mu x-y|^{n-\alpha}}
=\mu^\sigma\int_{R^n}\frac{\mu^n u^p(\mu z)dz}{|\mu(x-z)|^{n-\alpha}}\\[3mm]
&=\mu^\sigma\displaystyle\int_{R^n}\frac{\mu^{n-p\sigma}u_\mu^p(z)dz}
{\mu^{n-\alpha}|x-z|^{n-\alpha}}
=\mu^{\sigma-p\sigma+\alpha}
\displaystyle\int_{R^n}\frac{u_\mu^p(y)dy}{|x-y|^{n-\alpha}}.
\end{array}
$$
If $u_\mu$ still solves (\ref{3.1}), then
\begin{equation} \label{3.4}
\sigma=\frac{\alpha}{p-1}.
\end{equation}

Next,
$$
\int_{R^n}u_\mu^{p+1}(x)dx=\int_{R^n}[\mu^\sigma u(\mu x)]^{p+1}dx
=\mu^{\sigma(p+1)-n}\int_{R^n}u^{p+1}(z)dz.
$$
If the $L^{p+1}(R^n)$-norm is invariant, then there holds
$$
\sigma=\frac{n}{p+1}.
$$
Combining this with (\ref{3.4}), we get (\ref{3.2}).

On the contrary, if (\ref{3.2}) is true, then we can also deduce
the invariance by the same calculation above.
\end{proof}

\begin{theorem} \label{th4.2}
The Wolff type equation
\begin{equation} \label{3.3}
u(x)=\int_0^\infty(\frac{\int_{B_t(x)}u^p(y)dy}
{t^{n-\beta\gamma}})^{\frac{1}{\gamma-1}} \frac{dt}{t}
\end{equation}
and the energy $\|u\|_{L^{p+\gamma-1}(R^n)}$
are invariant under the scaling transform, if and only if
\begin{equation}
p=\frac{n+\beta\gamma}{n-\beta\gamma}(\gamma-1).
\label{3.5}
\end{equation}
In addition, (\ref{3.3}) and another energy $\|u\|_{L^{p+1}(R^n)}$
are invariant under the scaling transform, if and only if
\begin{equation}
p=\gamma^*-1 \quad (where \quad
\gamma^*=\frac{n\gamma}{n-\beta\gamma}). \label{3.5*}
\end{equation}
\end{theorem}

\begin{proof}
Take the scaling transform
$$
u_\mu(x)=\mu^\sigma u(\mu x).
$$
Then
$$\begin{array}{ll}
u_\mu(x)&=\mu^\sigma\displaystyle\int_0^\infty(\frac{\int_{B_t(\mu x)}u^p(y)dy}
{t^{n-\beta\gamma}})^{\frac{1}{\gamma-1}} \frac{dt}{t}\\[3mm]
&=\mu^\sigma\displaystyle\int_0^\infty(\frac{\int_{B_t(\mu x)} u^p(\mu z)d(\mu z)}
{t^{n-\beta\gamma}})^{\frac{1}{\gamma-1}} \frac{dt}{t}\\[3mm]
&=\mu^\sigma\displaystyle\int_0^\infty(\frac{\int_{B_s(x)}\mu^{n-p\sigma} u_\mu^p(z)dz}
{(\mu s)^{n-\beta\gamma}})^{\frac{1}{\gamma-1}} \frac{ds}{s}\\[3mm]
&=\mu^{\sigma+\frac{\beta\gamma-p\sigma}{\gamma-1}}
\displaystyle\int_0^\infty(\frac{\int_{B_s(x)}u_\mu^p(z)dz}
{s^{n-\beta\gamma}})^{\frac{1}{\gamma-1}} \frac{ds}{s}.
\end{array}
$$
Thus, $u_\mu$ solves (\ref{3.3}) if and only if
\begin{equation}
\sigma=\frac{\beta\gamma}{p-\gamma+1}.
\label{3.6}
\end{equation}

Next,
$$
\int_{R^n}u_\mu^{p+\gamma-1}(x)dx=\int_{R^n}[\mu^\sigma u(\mu x)]^{p+\gamma-1}dx
=\mu^{\sigma(p+\gamma-1)-n}\int_{R^n}u^{p+\gamma-1}(z)dz.
$$
The $L^{p+\gamma-1}(R^n)$-norm is invariant, if and only if
$$
\sigma=\frac{n}{p+\gamma-1}.
$$
Combining this with (\ref{3.6}), we get (\ref{3.5}).

At last,
$$
\int_{R^n}u_\mu^{p+1}(x)dx=\int_{R^n}[\mu^\sigma u(\mu x)]^{p+1}dx
=\mu^{\sigma(p+1)-n}\int_{R^n}u^{p+1}(z)dz.
$$
The $L^{p+1}(R^n)$-norm is invariant, if and only if
$$
\sigma=\frac{n}{p+1}.
$$
Combining this with (\ref{3.6}), we get (\ref{3.5*}).
\end{proof}

Since the corresponding result of the $\gamma$-Laplace equation
can not be covered by that of the Wolff type equation, we should
point out the following conclusion.

\begin{theorem} \label{th4.3}
The $\gamma$-Laplace equation
\begin{equation} \label{pL}
-\Delta_\gamma u(x)=u^p(x)
\end{equation}
and the energy $\|u\|_{L^{p+\gamma-1}(R^n)}$ are invariant under
the scaling transform, if and only if (\ref{3.5}) with $\beta=1$
holds. In addition, (\ref{pL}) and another energy
$\|u\|_{L^{p+1}(R^n)}$ are invariant under the scaling transform,
if and only if (\ref{3.5*}) with $\beta=1$ holds.
\end{theorem}

\begin{proof}
Suppose $u_\mu$ is a solution of (\ref{pL}). Then
$$
-div(|\nabla u_\mu|^{\gamma-2}\nabla u_\mu)=u_\mu^p.
$$
$$
-\mu^{\sigma(\gamma-1)}
div_x(|\nabla_x u(\mu x)|^{\gamma-2}\nabla_x u(\mu x))=\mu^{p\sigma}u^p(\mu x).
$$
Let $y=\mu x$, then
$$
-\mu^{\sigma(\gamma-1)+\gamma}
div_y(|\nabla_y u(y)|^{\gamma-2}\nabla_y u(y))=\mu^{p\sigma}u^p(y).
$$
This result shows that the equation is invariant if and only if
$$
\sigma=\frac{\gamma}{p-\gamma+1}.
$$
By the same argument as in Theorem \ref{th4.2},
the invariance of the energy is equivalent to
$$
\sigma=\frac{n}{p+\gamma-1}.
$$
Eliminating $\sigma$ from the two formulas above
yields $p=\frac{n+\gamma}{n-\gamma}(\gamma-1)$.

The proof that (\ref{3.5*}) with $\beta=1$ is the sufficient and
necessary condition is the same as the argument above.
\end{proof}

\subsection{HLS type equation}

\begin{theorem} \label{th4.4}
Assume $u>0$ is a classical solution of
\begin{equation} \label{hoPDE}
-\Delta u(x)=u^p(x), \quad x \in R^n.
\end{equation}
Assume $u \in L^{2^*}(R^n)$. Then $\nabla u \in L^2(R^n)$ if and only if
$u \in L^{p+1}(R^n)$.
\end{theorem}

A classical positive solution $u \in L^{2^*}(R^n)$ of
(\ref{hoPDE}) is called {\it finite energy solution}, if $u \in
L^{p+1}(R^n)$ or $\nabla u \in L^2(R^n)$.

\begin{proof}
Take smooth function $\zeta(x)$ satisfying
$$\begin{array}{lll}
&\zeta(x)=1, \quad &for~ |x| \leq 1;\\
&\zeta(x) \in [0,1], \quad &for~ |x| \in [1,2];\\
&\zeta(x)=0, \quad &for~ |x| \geq 2.
\end{array}
$$
Define the cut-off function
\begin{equation} \label{cut}
\zeta_R(x)=\zeta(\frac{x}{R}).
\end{equation}

Multiplying (\ref{hoPDE}) by $u\zeta_R^{2}$ and integrating on $D:=
B_{3R}(0)$, we have
$$
-\int_Du\zeta_R^{2}\Delta udx=\int_D u^{p+1}\zeta_R^{2}dx.
$$
Integrating by parts, we obtain
\begin{equation} \label{L1}
\int_D|\nabla u|^2\zeta_R^{2}dx+2\int_D u\zeta_R \nabla u
\nabla\zeta_R dx=\int_D u^{p+1}\zeta_R^{2}dx.
\end{equation}
Applying the Young inequality, we get
\begin{equation} \label{L2}
|\int_D u\zeta_R \nabla u
\nabla\zeta_R dx| \leq
\delta \int_D|\nabla u|^2 \zeta_R^{2}dx +C\int_D u^2 |\nabla\zeta_R|^2 dx
\end{equation}
for any $\delta \in (0,1/2)$.
If $u \in L^{2^*}(R^n)$, we can find $C>$ which is independent of
$R$ such that
\begin{equation} \label{L3}
\int_D u^2 |\nabla\zeta_R|^2 dx \leq C.
\end{equation}
If $u \in L^{p+1}(R^n)\cap L^{2^*}(R^n)$, then (\ref{L1})-(\ref{L3}) imply
$\int_D|\nabla u|^2\zeta_R^{2}dx \leq C$. Letting $R \to \infty$ yields
$$
\nabla u \in L^2(R^n).
$$
This and $u \in L^{2^*}(R^n)$ show that for some $R=R_j \to \infty$,
$$
R\int_{\partial D}(|\nabla u|^2+u^{2^*})ds \to 0,
$$
by Proposition \ref{prop2.1}. Therefore,
\begin{equation} \label{L4}
|\int_{\partial D}u \partial_{\nu}uds| \leq (\int_{\partial D}u^{2^*} ds)^{1/2^*}
(\int_{\partial D} |\nabla u|^2 ds)^{1/2} R^{(n-1)({1/2-1/2^*})} \to 0,
\end{equation}
when $R \to \infty$. Multiplying (\ref{hoPDE}) by $u$ yields
\begin{equation} \label{L5}
\int_D u^{p+1} dx=\int_D|\nabla u|^2dx-\int_{\partial D}u \partial_{\nu}uds.
\end{equation}
Letting $R \to \infty$ and using the result above, we have
$\|\nabla u\|_2^2=\|u\|_{p+1}^{p+1}$.

If $\nabla u \in L^2(R^n)$ and
$u \in L^{2^*}(R^n)$, (\ref{L4}) still holds.
If letting $R \to \infty$ in (\ref{L5})
and inserting (\ref{L4}) into it, we obtain
$\|u\|_{p+1}^{p+1}=\|\nabla u\|_2^2$ and hence
$u \in L^{p+1}(R^n)$.
\end{proof}

Next, we use the Pohozaev type identity in integral forms to
discuss the existence of the finite energy solutions of
(\ref{3.1}).
A positive classical solution $u$ of (\ref{3.1}) is called {\it finite
energy solution}, if $u \in L^{p+1}(R^n)$.

\begin{theorem} \label{th4.5}
The HLS type integral equation (\ref{3.1})
has positive classical solution in $L^{p+1}(R^n)$ if and only if
(\ref{3.2}) holds.
\end{theorem}

\begin{proof}
If (\ref{3.2}) holds, (\ref{3.1}) exists a unique class of finite energy solutions
(cf. \cite{ChLO} or \cite{Lieb}):
$$
u(x)=c(\frac{t}{t^2+|x-x_0|^2})^{(n-\alpha)/2}.
$$
Here $c,t$ are positive constants.

On the contrary,
if $u \in L^{p+1}(R^n)$ solves (\ref{3.1}), we claim that (\ref{3.2}) is true.
In fact, for any $\mu \neq 0$, from (\ref{3.1}) it follows
$$
u(\mu x)=\int_{R^n}\frac{u^p(y)dy}{|\mu x-y|^{n-\alpha}}
=\int_{R^n}\frac{\mu^nu^p(\mu z)dz}{|\mu (x-z)|^{n-\alpha}}
=\mu^{\alpha}\int_{R^n}\frac{u^p(\mu z)dz}{|x-z|^{n-\alpha}}.
$$
Differentiate both sides with respect to $\mu$. Then,
$$
x \cdot \nabla u(\mu x)=\alpha\mu^{\alpha-1}\int_{R^n}\frac{u^p(\mu z)dz}{|x-z|^{n-\alpha}}
+\mu^{\alpha}\int_{R^n}\frac{pu^{p-1}(\mu z)(z\cdot \nabla u)dz}{|x-z|^{n-\alpha}}.
$$
Letting $\mu=1$ yields
\begin{equation} \label{phzv}
x \cdot \nabla u(x)=\alpha u(x)
+\int_{R^n}\frac{z\cdot \nabla u^p(z)dz}{|x-z|^{n-\alpha}}.
\end{equation}

To handle the last term of the right hand side of (\ref{phzv}), we
integrate by parts to get
\begin{equation} \label{ibp}
\int_{B_R}\frac{z\cdot \nabla u^p(z)dz}{|x-z|^{n-\alpha}}
=R\int_{\partial B_R}\frac{u^p(z)ds}{|x-z|^{n-\alpha}}
-I_R(x)
\end{equation}
for any $R>0$. Here $B_R=B_R(0)$ and
$$
I_R(x)=n\int_{B_R}\frac{u^p(z)dz}{|x-z|^{n-\alpha}}+(n-\alpha)\int_{B_R}
\frac{(z\cdot(x-z))u^p(z)}{|x-z|^{{n-\alpha}+2}}dz.
$$

Next, we claim the first term of the right hand side of (\ref{ibp}) converges
to zero as $R \to \infty$. In fact, for suitably large $R$,
\begin{equation} \label{L6}
\begin{array}{ll}
&\quad R\displaystyle\int_{\partial B_R}\frac{u^p(z)ds}{|x-z|^{n-\alpha}}\\[3mm]
&\leq CR^{1+\alpha-n} (R\displaystyle\int_{\partial B_R}u^{p+1} ds)^{\frac{p}{p+1}}
R^{-\frac{p}{p+1}}R^{\frac{n-1}{p+1}}\\[3mm]
&=CR^{\alpha-n+\frac{n}{p+1}}(R\displaystyle\int_{\partial B_R}u^{p+1} ds)^{\frac{p}{p+1}}.
\end{array}
\end{equation}
By Theorem \ref{th2.3}, we see that $p \geq \frac{n}{n-\alpha}$.
So $\alpha-n+\frac{n}{p+1}<0$.
In addition, using Proposition 2.1 and $u \in L^{p+1}(R^n)$,
we can find $R_j \to \infty$ such that
\begin{equation} \label{L7}
R_j\int_{\partial B_{R_j}} u^{p+1}ds \to 0.
\end{equation}
Let $R=R_j \to \infty$ in (\ref{L6}), we verify our claim.

Multiplying (\ref{phzv}) by $u^p(x)$ and applying the claim above, we obtain
$$\begin{array}{ll}
&\quad\displaystyle\int_{R^n}u^p(x)(x\cdot \nabla u(x))dx\\[3mm]
&=\alpha\displaystyle\int_{R^n}u^{p+1}(x)dx
+\int_{R^n}u^p(x)dx\int_{R^n}\frac{z\cdot \nabla u^p(z)dz}{|x-z|^{n-\alpha}}\\[3mm]
&=\alpha\displaystyle\int_{R^n}u^{p+1}(x)dx
-n\int_{R^n}u^p(x)dx\int_{R^n}\frac{u^p(z)dz}{|x-z|^{n-\alpha}}\\[3mm]
&\quad-(n-\alpha)\displaystyle\int_{R^n}
\int_{R^n}\frac{(z\cdot(x-z))u^p(x)u^p(z)}{|x-z|^{{n-\alpha}+2}}dzdx.
\end{array}
$$
By virtue of $z\cdot(x-z)+x\cdot(z-x)=-|x-z|^2$, it follows that
$$\begin{array}{ll}
&\quad\displaystyle
\int_{R^n}u^p(x)(x\cdot \nabla u(x))dx\\[3mm]
&=\alpha\displaystyle\int_{R^n}u^{p+1}(x)dx
-n\int_{R^n}u^{p+1}(x)dx\\[3mm]
&\quad +\displaystyle\frac{n-\alpha}{2}\int_{R^n}
\int_{R^n}\frac{u^p(x)u^p(z)}{|x-z|^{{n-\alpha}}}dzdx\\[3mm]
&=-\displaystyle\frac{n-\alpha}{2}\int_{R^n}u^{p+1}(x)dx.
\end{array}
$$
On the other hand, integrating by parts an using (\ref{L7}),
we get
$$
\int_{R^n}u^p(x)(x\cdot \nabla u(x))dx
=\frac{1}{p+1}\int_{R^n}(x\cdot\nabla u^{p+1}(x))dx
=\frac{-n}{p+1}\int_{R^n}u^{p+1}(x)dx.
$$
Combining this with the result above, we deduce that
$$
\frac{1}{p+1}=\frac{n-\alpha}{2n}.
$$
This is (\ref{3.2}). Theorem \ref{th4.5} is proved.
\end{proof}

\begin{corollary} \label{coro4.6}
Let $k \in [1,n/2)$ be an integer and $p>1$.
The $2k$-order Lane-Emden PDE
\begin{equation} \label{2kPDE}
(-\Delta)^k u(x)=u^p(x), \quad u>0 ~in ~R^n,
\end{equation}
has positive classical solution in $L^{p+1}(R^n)$ if and only if
$p=\frac{n+2k}{n-2k}$.
\end{corollary}

\begin{proof}
When $p>1$, Corollary \ref{coro2.4} shows that (\ref{2kPDE}) is equivalent to
the HLS type equation (\ref{3.1}) with
$\alpha=2k$. According to Theorem \ref{th4.5}, we have the corresponding critical
conditions $p=\frac{n+2k}{n-2k}$ for the existence of the finite
energy solutions of the (\ref{hoPDE}).
\end{proof}

\paragraph{Remark 4.1.}
Theorem \ref{th4.5} shows another critical condition (\ref{3.2}) for the existence
of the positive solutions to (\ref{3.1}). Since the finite energy solutions class of
(\ref{3.1}) is smaller than the positive solutions class of (\ref{HLS}),
the critical condition (\ref{3.2}) is stronger than (\ref{cc1}).

\subsection{$\gamma$-Laplace equation}

Serrin and Zou \cite{SZ} proved that
$\gamma$-Laplace equation has positive classical solutions if and only if $p
\geq \gamma^*-1$, where $\gamma^*=\frac{n\gamma}{n-\gamma}$.
Naturally, we conjecture that $\gamma$-Laplace
equation has the finite energy solution if and only if
$p=\gamma^*-1$.

To define the finite energy solution, we first introduce the following
theorem. It is a natural generalization of Theorem \ref{th4.4}.

\begin{theorem} \label{th4.7}
Assume $u>0$ is a classical solution of the
$\gamma$-Laplace equation (\ref{pL}). Assume $u \in L^{\gamma^*}(R^n)$
with $\gamma^*=\frac{n\gamma}{n-\gamma}$.
Then $\nabla u \in L^\gamma(R^n)$ if and only if
$u \in L^{p+1}(R^n)$. In addition,
$\|\nabla u\|_\gamma^\gamma=\|u\|_{p+1}^{p+1}$.
\end{theorem}

A classical positive solution $u \in L^{\gamma^*}(R^n)$ of (\ref{pL})
is called {\it finite energy solution}
if $u \in L^{p+1}(R^n)$ or $\nabla u \in L^\gamma(R^n)$.

\begin{proof}
Let $u \in L^{\gamma^*}(R^n)$. Take a cut-off function $\zeta_R$
as (\ref{cut}). Using the H\"older inequality, we get
\begin{equation} \label{asd}
\int_Du^\gamma |\nabla\zeta_R|^\gamma dx \leq \|u\|_{\gamma^*,D}^\gamma
\|\nabla \zeta\|_{\frac{\gamma\gamma^*}{\gamma^*-\gamma},D}^\gamma
\leq C,
\end{equation}
where $D=B_{2R}(0)$, and $C>0$ is independent of $R$.

(1) Sufficiency. Supposing $u \in L^{p+1}(R^n) \cap L^{\gamma^*}(R^n)$ solves (\ref{pL}),
we claim $\nabla u \in L^\gamma(R^n)$ and $\|\nabla u\|_\gamma^\gamma=\|u\|_{p+1}^{p+1}$.

Multiplying (\ref{pL}) by $u\zeta_R^\gamma$ and integrating by parts on $D$,
we obtain
\begin{equation} \label{cheng}
\int_D|\nabla u|^\gamma\zeta_R^\gamma dx+\gamma\int_D|\nabla u|^{\gamma-2}
(u\zeta_R^{\gamma-1})\nabla u \nabla \zeta_R dx
=\int_Du^{p+1} \zeta_R^\gamma dx.
\end{equation}
Using the Young inequality,
from (\ref{cheng}) we deduce that for any $\delta \in (0,1/2)$,
$$
\begin{array}{ll}
\displaystyle\int_D|\nabla u|^\gamma \zeta_R^\gamma dx
&\leq C|\displaystyle\int_D|\nabla u|^{\gamma-2}
(u\zeta_R^{\gamma-1})\nabla u \nabla \zeta_R dx|+\int_Du^{p+1} \zeta_R^\gamma dx\\[3mm]
&\leq \delta \displaystyle\int_D|\nabla u|^\gamma \zeta_R^\gamma dx
+C\int_D u^\gamma |\nabla\zeta_R|^\gamma dx
+\int_Du^{p+1} \zeta_R^\gamma dx.
\end{array}
$$
Combining this result with (\ref{asd}), we see that
$\int_D|\nabla u|^\gamma \zeta_R^\gamma dx \leq C$. Let $R \to \infty$, then
$$
\nabla u \in L^\gamma(R^n).
$$
From this result as well as $u \in L^{\gamma^*}(R^n)$, we use Proposition
\ref{prop2.1} to deduce that for some $R_j$ (denoted by $R$),
\begin{equation} \label{cvb}
R\int_{\partial D}(|\nabla u|^\gamma +u^{\gamma^*})ds <o(1),
\quad as ~R \to \infty.
\end{equation}
Multiplying (\ref{pL}) by $u$ and integrating on $D$, we have
\begin{equation} \label{vbn}
\int_D u^{p+1}dx=-\int_D u \Delta_\gamma udx
=\int_D|\nabla u|^\gamma dx -\int_{\partial D} u|\nabla u|^{\gamma-2}
\partial_{\nu} uds.
\end{equation}
By means of the H\"older inequality and (\ref{cvb}), we get
$$\begin{array}{ll}
&\quad |\displaystyle\int_{\partial D} u|\nabla u|^{\gamma-2}
\partial_{\nu} uds|\\[3mm]
&\leq [R\displaystyle\int_{\partial D}|\nabla u|^\gamma ds]^{\frac{\gamma-1}{\gamma}}
R^{\frac{1-\gamma}{\gamma}}[R\int_{\partial D} u^{\gamma^*}ds]^{1/\gamma^*}R^{-1/\gamma^*}
R^{(n-1)(1/\gamma-1/\gamma^*)}\\[3mm]
&< o(1).
\end{array}
$$
Letting $R \to \infty$ in (\ref{vbn}),
we have
\begin{equation} \label{guo}
\int_{R^n}|\nabla u|^\gamma dx=\int_{R^n}u^{p+1}dx.
\end{equation}

(2) Necessity.
Let $u >0$ solve (\ref{pL}).
If $\nabla u \in L^\gamma(R^n)$ and $u \in L^{\gamma^*}(R^n)$, then
(\ref{cvb}) is still true. Using (\ref{cvb}) to handle the last term of
the right hand side of (\ref{vbn}), we also derive (\ref{guo}) and hence
$u \in L^{p+1}(R^n)$.

\end{proof}

\begin{theorem} \label{th4.8}
The $\gamma$-Laplace equation (\ref{pL}) has a classical solution
satisfying $\nabla u \in L^{\gamma}(R^n)$ if and only if
\begin{equation} \label{pLc}
p=\gamma^*-1.
\end{equation}
\end{theorem}

\begin{proof}
If $p=\gamma^*-1$, according to p.328 in \cite{Gazzola}, (\ref{pL}) admits
a class of solutions
$$
u(x)=\frac{d}{[1+D(d^{\frac{\gamma}{n-\gamma}}
|x|^{\frac{\gamma}{\gamma-1}})]^{\frac{n-\gamma}{\gamma}}}.
$$
Here $d,D$ are positive constants.

Next, we prove the sufficiency.
Write $B=B_R(0)$. Multiplying the equation with $(x\cdot \nabla u)$ and integrating on $B$,
we obtain
$$\begin{array}{ll}
&\displaystyle
\int_B|\nabla u|^{\gamma-2}\nabla u\nabla(x\cdot\nabla u)dx
-\int_{\partial B}|\nabla u|^{\gamma-2}(\nu\cdot\nabla u)(x\cdot\nabla u)ds\\[3mm]
&=\displaystyle\int_B u^p(x\cdot\nabla u)dx.
\end{array}
$$
Here $\nu$ is the unit outward normal vector to $\partial B$.
Noting
$$
\nabla u\nabla(x\cdot\nabla u)=|\nabla u|^2+\frac{1}{2}x\cdot\nabla(|\nabla u|^2)
$$
and $x=|x|\nu$, we have
$$\begin{array}{ll}
&\displaystyle\int_B|\nabla u|^\gamma dx+\frac{1}{\gamma}\int_B x\cdot\nabla(|\nabla u|^\gamma)dx
-R\int_{\partial B}|\nabla u|^{\gamma-2}|\partial_\nu u|^2ds\\[3mm]
&=\displaystyle\frac{1}{p+1}\int_B x\cdot\nabla u^{p+1}dx.
\end{array}
$$
Integrating by parts, we get
\begin{equation} \label{deng}
\begin{array}{ll}
&\displaystyle(1-\frac{n}{\gamma})\int_B|\nabla u|^\gamma dx
+\frac{R}{\gamma}\int_{\partial B}|\nabla u|^\gamma ds
-R\int_{\partial B}|\nabla u|^{\gamma-2}|\partial_\nu u|^2 ds\\[3mm]
&=\displaystyle\frac{R}{p+1}\int_{\partial B}u^{p+1}ds-\frac{n}{p+1}\int_Bu^{p+1}dx.
\end{array}
\end{equation}
According to Theorem \ref{th4.7}, $\nabla u \in L^\gamma(R^n)$
implies $u \in L^{\gamma^*}(R^n) \cap L^{p+1}(R^n)$.
Therefore, by Proposition \ref{prop2.1}, we can find $R_j \to \infty$, such that
$$
R_j\int_{\partial B_{R_j}}(u^{p+1}+|\nabla u|^\gamma)ds \to 0.
$$
Let $R=R_j \to \infty$ in (\ref{deng}). By means of the result above, we deduce that
$$
(1-\frac{n}{\gamma})\int_{R^n}|\nabla u|^\gamma dx
=-\frac{n}{p+1}\int_{R^n}u^{p+1}dx.
$$
Inserting (\ref{guo}) into this result yields $p=\gamma^*-1$.
\end{proof}

\paragraph{Remark 4.2.}
\begin{enumerate}
\item For the Wolff type equation (\ref{3.3}),
we do not know whether (\ref{3.5})
is the necessary and sufficient condition for the existence of positive
solution in $L^{p+\gamma-1}(R^n)$.

\item A surprising observation is, when $\gamma \neq 2$,
the critical condition (\ref{pLc})
is different from (\ref{3.5}) with $\beta=1$. One reason
is that the solution of (\ref{pL}) only solves a Wolff type equation with
variable coefficient instead of (\ref{3.3}). Another reason is that the
finite energy functions spaces $L^{p+1}(R^n)$ and
$L^{p+\gamma-1}(R^n)$ are also different except for $\gamma=2$.
This distinction shows that
(\ref{3.2}) and (\ref{pLc}) are not the same class critical exponents.
For $\gamma$-Laplace equation, besides the divided number in Theorem
\ref{th2.6}, we also have two critical exponents mentioned above.
The relation of them is
$$
\frac{n(\gamma-1)}{n-\gamma}<\frac{n+\gamma}{n-\gamma}(\gamma-1)
<\gamma^*-1
$$
as long as $\gamma \in (1,2)$.
This is also led to by the difference of the
existence spaces of positive solutions.
\end{enumerate}

\section{Finite energy solutions: system}

\subsection{Critical conditions and scaling invariants}

\begin{theorem} \label{th5.1}
(1) Both the semilinear Lane-Emden type system
\begin{equation}
 \left \{
   \begin{array}{l}
      -\Delta u=v^q\\
      -\Delta v=u^p.
   \end{array}
   \right.  \label{L-E}        
 \end{equation}
and the energy integrals $\|u\|_{L^{p+1}(R^n)}$ and $\|v\|_{L^{q+1}(R^n)}$
are invariant under the scaling transforms, if and only if
\begin{equation} \label{L-Ec}
\frac{1}{p+1}+\frac{1}{q+1}=\frac{n-2}{n}.
\end{equation}

(2) Both the $\gamma$-Laplace system
\begin{equation}
 \left \{
   \begin{array}{l}
      -\Delta_\gamma u=v^q,\\
      -\Delta_\gamma v=u^p.
   \end{array}
   \right.  \label{pL-s}        
 \end{equation}
and the energy integrals $\|u\|_{L^{p+\gamma-1}(R^n)}$ and
$\|v\|_{L^{q+\gamma-1}(R^n)}$
are invariant under the scaling transforms, if and only if
\begin{equation} \label{pL-c}
\frac{1}{p+\gamma-1}+\frac{1}{q+\gamma-1}=\frac{n-\gamma}{n(\gamma-1)}.
\end{equation}
In addition, (\ref{pL-s}) and the energy integrals
$\|u\|_{L^{p+1}(R^n)}$ and $\|v\|_{L^{q+1}(R^n)}$ are invariant
under the scaling transforms, if and only if
\begin{equation} \label{pL-cc}
p=q \quad or \quad \gamma=2.
\end{equation}

(3) The HLS type system
\begin{equation}
 \left \{
   \begin{array}{l}
      u(x)=\displaystyle\int_{R^n}\frac{v^q(y)dy}{|x-y|^{n-\alpha}},\\
      v(x)=\displaystyle\int_{R^n}\frac{u^p(y)dy}{|x-y|^{n-\alpha}}
   \end{array}
   \right. \label{hls}         
 \end{equation}
and the energy integrals $\|u\|_{L^{p+1}(R^n)}$ and $\|v\|_{L^{q+1}(R^n)}$
are invariant under the scaling transforms, if and only if
\begin{equation} \label{hlsc}
\frac{1}{p+1}+\frac{1}{q+1}=\frac{n-\alpha}{n}.
\end{equation}

(4) The Wolff type system
\begin{equation}
 \left \{
   \begin{array}{l}
      u(x)=W_{\beta,\gamma}(v^q)(x),\\
      v(x)=W_{\beta,\gamma}(u^p)(x).
   \end{array}
   \right. \label{wfs}         
 \end{equation}
and the energy integrals $\|u\|_{L^{p+\gamma-1}(R^n)}$ and $\|v\|_{L^{q+\gamma-1}(R^n)}$
are invariant under the scaling transforms, if and only if
\begin{equation} \label{wfsc}
\frac{1}{p+\gamma-1}+\frac{1}{q+\gamma-1}=\frac{n-\beta\gamma}{n(\gamma-1)}.
\end{equation}
In addition, (\ref{wfs}) and the energy integrals
$\|u\|_{L^{p+1}(R^n)}$ and $\|v\|_{L^{q+1}(R^n)}$ are invariant
under the scaling transforms, if and only if
\begin{equation} \label{wfscc}
p=q \quad or \quad \gamma=2.
\end{equation}
\end{theorem}

\begin{proof}
(1) Take the scaling transforms
\begin{equation} \label{scale}
u_\mu(x)=\mu^{\sigma_1} u(\mu x), \quad v_\mu(x)=\mu^{\sigma_2}
v(\mu x).
\end{equation}
Then
$$
-\Delta_x u_\mu(x)=\mu^{\sigma_1+2}[-\Delta_yu(y)]
=\mu^{\sigma_1+2}v^q(y)=\mu^{\sigma_1+2-q\sigma_2}v_\mu^q(x),
$$
and similarly,
$$
-\Delta v_\mu=\mu^{\sigma_2+2-p\sigma_1}u_\mu^p.
$$
On the other hand,
$$
\int_{R^n}u_\mu^{p+1}(x)dx=\mu^{\sigma_1(p+1)}\int_{R^n}u^{p+1}(\mu x)dx
=\mu^{\sigma_1(p+1)-n}\int_{R^n}u^{p+1}(y)dy,
$$
and similarly,
$$
\int_{R^n}v_\mu^{q+1}(x)dx
=\mu^{\sigma_2(q+1)-n}\int_{R^n}v^{q+1}(y)dy.
$$
Clearly, (\ref{L-E}) is invariant if and only if
$$
\sigma_1+2=q\sigma_2, \quad \sigma_2+2=p\sigma_1.
$$
Energy integrals are invariant if and only if
$$
\sigma_1(p+1)=n, \quad \sigma_2(q+1)=n.
$$
Eliminate $\sigma_1$ and $\sigma_2$. Then
$$
\frac{pq-1}{(p+1)(q+1)}=\frac{2}{n}.
$$
This is equivalent to (\ref{L-Ec}).

(2) In view of (\ref{scale}), we have
$$
-\Delta_\gamma u_\mu(x)=\mu^{\sigma_1(\gamma-1)+\gamma} [-\Delta_\gamma u(\mu x)]
=\mu^{\sigma_1(\gamma-1)+\gamma-q\sigma_2}v_\mu^q(x).
$$
Similarly,
$$
-\Delta_\gamma v_\lambda=\mu^{\sigma_2(\gamma-1)+\gamma-p\sigma_1}u_\mu^p(x).
$$
In addition,
$$\begin{array}{ll}
\displaystyle\int_{R^n} u_\mu^{p+\gamma-1}(x)dx
&=\mu^{\sigma_1(p+\gamma-1)}\displaystyle\int_{R^n}u^{p+\gamma-1}(\mu x)dx\\[3mm]
&=\mu^{\sigma_1(p+\gamma-1)-n}\displaystyle\int_{R^n}u^{p+\gamma-1}(y)dy,
\end{array}
$$
and similarly,
$$
\int_{R^n} v_\mu^{q+\gamma-1}(x)dx
=\mu^{\sigma_2(q+\gamma-1)-n}\int_{R^n}v^{q+\gamma-1}(y)dy.
$$
Eq. (\ref{pL-s}) is invariant if and only if
$$
\sigma_1(\gamma-1)+\gamma-q\sigma_2=0,\quad
\sigma_2(\gamma-1)+\gamma-p\sigma_1=0.
$$
Namely,
\begin{equation} \label{daxue}
\sigma_1=\frac{\gamma(q+\gamma-1)}{pq-(\gamma-1)^2},
\quad \sigma_2=\frac{\gamma(p+\gamma-1)}{pq-(\gamma-1)^2}.
\end{equation}
Energy integrals $\|u\|_{L^{p+\gamma-1}(R^n)}$ and
$\|v\|_{L^{q+\gamma-1}(R^n)}$ are invariant if and only if
$$
\sigma_1(p+\gamma-1)-n=0, \quad \sigma_2(q+\gamma-1)-n=0.
$$
Eliminating $\sigma_1$ and $\sigma_2$, we obtain (\ref{pL-c}).

Similarly, $\|u\|_{L^{p+1}(R^n)}$ and $\|v\|_{L^{q+1}(R^n)}$ are
invariant if and only if
$$
\sigma_1=\frac{n}{p+1}, \quad \sigma_2=\frac{n}{q+1}.
$$
Combining with (\ref{daxue}), we see
$(q+1)(p+\gamma-1)=(p+1)(q+\gamma-1)$. This is equivalent to
$(p-q)(\gamma-2)=0$. Thus, (\ref{pL-cc}) is the sufficient and
necessary condition.

(3)
Noting (\ref{scale}), we have
$$\begin{array}{ll}
v_\mu(x)&=\mu^{\sigma_2}\displaystyle\int_{R^n}\frac{u^p(y)dy}{|\mu x-y|^{n-\alpha}}
=\mu^{\sigma_2}\int_{R^n}\frac{\mu^n u^p(\mu z)dz}{|\mu(x-z)|^{n-\alpha}}\\[3mm]
&=\mu^{\sigma_2}\displaystyle\int_{R^n}\frac{\mu^{n-p\sigma_1}u_\mu^p(z)dz}
{\mu^{n-\alpha}|x-z|^{n-\alpha}}
=\mu^{\sigma_2-p\sigma_1+\alpha}
\displaystyle\int_{R^n}\frac{u_\mu^p(y)dy}{|x-y|^{n-\alpha}}.
\end{array}
$$
Similarly,
$$
u_\mu(x)=\mu^{\sigma_1-q\sigma_2+\alpha}
\int_{R^n}\frac{v_\mu^q(y)dy}{|x-y|^{n-\alpha}}.
$$
Thus, $u_\mu,v_\mu$ still solve (\ref{hls}) if and only if
$$
\sigma_1+\alpha=q\sigma_2, \quad \sigma_2+\alpha=p\sigma_1.
$$
By the same calculation in (1), energy integrals are
invariant if and only if
$$
\sigma_1(p+1)=n, \quad \sigma_2(q+1)=n.
$$
Eliminating $\sigma_1$ and $\sigma_2$, we deduce (\ref{hlsc}).

(4)
Noting (\ref{scale}), we have
$$\begin{array}{ll}
v_\mu(x)&=\mu^{\sigma_2}\displaystyle\int_0^\infty(\frac{\int_{B_t(\mu x)}u^p(y)dy}
{t^{n-\beta\gamma}})^{\frac{1}{\gamma-1}} \frac{dt}{t}\\[3mm]
&=\mu^{\sigma_2}\displaystyle\int_0^\infty(\frac{\int_{B_t(\mu x)} u^p(\mu z)d(\mu z)}
{t^{n-\beta\gamma}})^{\frac{1}{\gamma-1}} \frac{dt}{t}\\[3mm]
&=\mu^{\sigma_2}\displaystyle\int_0^\infty(\frac{\int_{B_s(x)}\mu^{n-p\sigma_1} u_\mu^p(z)dz}
{(\mu s)^{n-\beta\gamma}})^{\frac{1}{\gamma-1}} \frac{ds}{s}\\[3mm]
&=\mu^{\sigma_2+\frac{\beta\gamma-p\sigma_1}{\gamma-1}}
\displaystyle\int_0^\infty(\frac{\int_{B_s(x)}u_\mu^p(z)dz}
{s^{n-\beta\gamma}})^{\frac{1}{\gamma-1}} \frac{ds}{s}.
\end{array}
$$
Similarly,
$$
u_\mu(x)=\mu^{\sigma_1+\frac{\beta\gamma-q\sigma_2}{\gamma-1}}
\displaystyle\int_0^\infty(\frac{\int_{B_s(x)}v_\mu^q(z)dz}
{s^{n-\beta\gamma}})^{\frac{1}{\gamma-1}} \frac{ds}{s}.
$$
Thus, $u_\mu,v_\mu$ still solve (\ref{wfs}) if and only if
$$
(\gamma-1)\sigma_1+\beta\gamma=q\sigma_2,
\quad (\gamma-1)\sigma_2+\beta\gamma=p\sigma_1.
$$
By the same calculation in (2), energy integrals
$\|u\|_{L^{p+\gamma-1}(R^n)}$ and $\|v\|_{L^{q+\gamma-1}(R^n)}$
are invariant if and only if
$$
\sigma_1(p+\gamma-1)=n, \quad \sigma_2(q+\gamma-1)=n.
$$
Eliminating $\sigma_1$ and $\sigma_2$, we deduce (\ref{wfsc}).

By the same argument in (2), (\ref{wfscc}) is another
corresponding sufficient and necessary condition.
\end{proof}

\subsection{Existence and the critical conditions}

In this subsection, we first show that (\ref{L-Ec}) is the critical condition
of the existence of the finite energy solution of (\ref{L-E}).
We call the positive classical solutions $u,v$ of (\ref{L-E}) {\it
finite energy solutions}, if $u \in L^{p+1}(R^n)
\cap L^{2^*}(R^n)$, and $v \in L^{q+1}(R^n) \cap L^{2^*}(R^n)$.

\begin{theorem} \label{th5.2}
The system (\ref{L-E}) has a pair of finite energy solutions $(u,v)$
if and only if (\ref{L-Ec}) holds.
\end{theorem}

\begin{proof}
Serrin and Zou \cite{SZ-1998} proved the existence if (\ref{L-Ec}) is true.
Next, we will deduce (\ref{L-Ec}) from the existence. Denote $B_R(0)$ by $B$.
According to Proposition 5.1 in \cite{SZ-DIE} (or cf. Lemma 2.6 in \cite{Souplet}),
the solutions $u,v$ satisfy the Pohozaev
type identity
\begin{equation}
\begin{array}{ll}
&(\displaystyle\frac{n}{p+1}-a_1)\int_B u^{p+1}dx+(\frac{n}{q+1}-a_2)\int_B v^{q+1}dx\\[3mm]
&=R^n\displaystyle\int_{S^{n-1}}(\frac{u^{p+1}}{p+1}+\frac{v^{q+1}}{q+1})ds
+R^{n-1}\int_{S^{n-1}}(a_1u\partial_rv+a_2v\partial_ru)ds\\[3mm]
&\quad +R^n\displaystyle\int_{S^{n-1}}(\partial_ru\partial_rv
-\frac{\partial_\theta u\partial_\theta v}{R^2})ds,
\end{array}
\label{PR}
\end{equation}
where $a_2,a_2 \in R$ satisfy $a_1+a_2=n-2$. Since $u,v$ are
finite energy solutions, we know $\nabla u,\nabla v \in L^2(R^n)$ by an
analogous argument of Theorem \ref{th4.4}.
Using Proposition \ref{prop2.1} and the
Young inequality, we can find
$R_j \to \infty$, such that all the terms in the right hand side converge to zero.
Letting $R=R_j \to \infty$ in the Pohozaev identity above,
we obtain
$$
(\frac{n}{p+1}-a_1)\int_{R^n}
u^{p+1}dx+(\frac{n}{q+1}-a_2)\int_{R^n} v^{q+1}dx=0
$$
for any $a_1,a_2$ as long as $a_1+a_2=n-2$.
Take $a_2=\frac{n}{q+1}$, then
$$
(\frac{n}{p+1}-a_1)\int_{R^n} u^{p+1}dx=0.
$$
This implies $0=\frac{n}{p+1}-a_1=\frac{n}{p+1}-(n-2-a_2)
=\frac{n}{p+1}-(n-2)+\frac{n}{q+1}$. So (\ref{L-Ec}) is verified.
\end{proof}

Next, we consider the HLS type system. Since (\ref{hls}) is the
Euler-Lagrange system of the extremal functions of the HLS inequality
which implies $(u,v) \in L^{p+1}(R^n) \times L^{q+1}(R^n)$,
we naturally call such solutions (belonging to $L^{p+1}(R^n) \times L^{q+1}(R^n)$)
of (\ref{hls}) as {\it finite energy solutions}.

\begin{theorem} \label{th5.3}
The HLS type system (\ref{hls}) has the finite energy solutions if and only if
(\ref{hlsc}) holds.
\end{theorem}

\begin{proof}
{\it Sufficiency.} Clearly, the extremal functions of the HLS inequality
are the finite energy solutions. Lieb \cite{Lieb} obtained the
existence of those extremal functions.

{\it Necessity.} The Pohozaev type identity in integral forms is used here.

For any $\mu \neq 0$, there holds
$$
v(\mu x)=\int_{R^n}\frac{u^p(y)dy}{|\mu x-y|^{n-\alpha}}
=\mu^{\alpha}\int_{R^n}\frac{u^p(\mu z)dz}{|x-z|^{n-\alpha}}.
$$
Differentiate both sides with respect to $\mu$ and let $\mu=1$. Then,
\begin{equation} \label{L8}
x \cdot \nabla v=\alpha v
+\int_{R^n}\frac{z\cdot \nabla u^p(z)dz}{|x-z|^{n-\alpha}}.
\end{equation}

According to Remark 1.2 (1) (or cf. Theorem 1 in \cite{CL-DCDS}),
if $p,q \leq \frac{\alpha}{n-\alpha}$, (\ref{hls}) has no any positive
solution. Therefore, $(u,v)$ solves (\ref{hls}) implies
$p,q>\frac{\alpha}{n-\alpha}$.
Similar to the derivation of (\ref{L6}), if follows
$$
R\int_{\partial B_R}\frac{u^p(z)ds}{|x-z|^{n-\alpha}} \to 0,
\quad R\int_{\partial B_R}\frac{v^q(z)ds}{|x-z|^{n-\alpha}} \to 0,
$$
when $R=R_j \to \infty$. Thus, integrating by parts, we obtain
$$
\int_{R^n}\frac{z\cdot \nabla u^p(z)dz}{|x-z|^{n-\alpha}}
=-nv-(n-\alpha)
\int_{R^n}\frac{(z\cdot(x-z))u^p(z)}{|x-z|^{{n-\alpha}+2}}dz.
$$

Multiplying (\ref{L8}) by $v^q(x)$ we get
$$\begin{array}{ll}
&\quad\displaystyle\int_{R^n}v^q(x)(x\cdot \nabla v(x))dx\\[3mm]
&=\alpha\displaystyle\int_{R^n}v^{q+1}(x)dx
+\int_{R^n}v^q(x)dx\int_{R^n}\frac{z\cdot \nabla u^p(z)dz}{|x-z|^{n-\alpha}}\\[3mm]
&=\alpha\displaystyle\int_{R^n}v^{q+1}(x)dx
-n\int_{R^n}v^{q+1}(x)dx\\[3mm]
&\quad-(n-\alpha)\displaystyle\int_{R^n}
\int_{R^n}\frac{(z\cdot(x-z))v^q(x)u^p(z)}{|x-z|^{{n-\alpha}+2}}dzdx.
\end{array}
$$
Similarly, there also holds
$$\begin{array}{ll}
&\quad\displaystyle\int_{R^n}u^p(x)(x\cdot \nabla u(x))dx\\[3mm]
&=(\alpha-n)\displaystyle\int_{R^n}u^{p+1}(x)dx
-(n-\alpha)\displaystyle\int_{R^n}
\int_{R^n}\frac{(z\cdot(x-z))v^q(z)u^p(x)}{|x-z|^{{n-\alpha}+2}}dzdx\\[3mm]
&=(\alpha-n)\displaystyle\int_{R^n}u^{p+1}(x)dx
-(n-\alpha)\displaystyle\int_{R^n}
\int_{R^n}\frac{(x\cdot(z-x))v^q(x)u^p(z)}{|x-z|^{{n-\alpha}+2}}dzdx.
\end{array}
$$
By virtue of $z\cdot(x-z)+x\cdot(z-x)=-|x-z|^2$, it follows that
$$\begin{array}{ll}
&\quad\displaystyle
\int_{R^n}v^q(x)(x\cdot \nabla v(x))dx
+\int_{R^n}u^p(x)(x\cdot \nabla u(x))dx\\[3mm]
&=(\alpha-n)(\displaystyle\int_{R^n}v^{q+1}(x)dx
+\int_{R^n}u^{p+1}(x)dx)\\[3mm]
&\quad +(n-\alpha)\displaystyle\int_{R^n}
\int_{R^n}\frac{v^q(x)u^p(z)}{|x-z|^{{n-\alpha}}}dzdx.
\end{array}
$$
On the other hand, integrating by parts leads to
$$\begin{array}{ll}
\displaystyle\int_{R^n}v^q(x)(x\cdot \nabla v(x))dx
&=\displaystyle\frac{1}{q+1}\int_{R^n}(x\cdot\nabla v^{q+1}(x))dx\\[3mm]
&=\displaystyle\frac{-n}{q+1}\int_{R^n}v^{q+1}(x)dx
\end{array}
$$
and similarly $\int_{R^n}u^p(x)(x\cdot \nabla u(x))dx
=\frac{-n}{p+1}\int_{R^n}u^{p+1}(x)dx$.
Inserting these into the result above, we deduce that
$$\begin{array}{ll}
&\quad-\displaystyle
\frac{n}{q+1}\int_{R^n}v^{q+1}(x)dx
-\frac{n}{p+1}\int_{R^n}u^{p+1}(x)dx\\[3mm]
&=(\alpha-n)(\displaystyle\int_{R^n}v^{q+1}(x)dx
+\int_{R^n}u^{p+1}(x)dx)\\[3mm]
&\quad +(n-\alpha)\displaystyle\int_{R^n}
\int_{R^n}\frac{v^q(x)u^p(z)}{|x-z|^{{n-\alpha}}}dzdx.
\end{array}
$$
From (\ref{hls}), it follows that
$$\begin{array}{ll}
&\displaystyle\int_{R^n}v^{q+1}(x)dx
=\int_{R^n}v^q(x)dx\int_{R^n}\frac{u^p(y)dy}{|x-y|^{n-\alpha}}\\[3mm]
&=\displaystyle\int_{R^n}u^p(x)dx\int_{R^n}\frac{v^q(y)dy}{|x-y|^{n-\alpha}}
=\int_{R^n}u^{p+1}(x)dx.
\end{array}
$$
Substituting this into the result above yields
$$
\frac{1}{p+1}+\frac{1}{q+1}=\frac{n-\alpha}{n}.
$$
Theorem \ref{th5.3} is proved.
\end{proof}

\begin{corollary} \label{coro5.4}
Let $k \in [1,n/2)$ be an integer and $pq>1$. The $2k$-order system
$$
 \left \{
   \begin{array}{l}
      (-\Delta)^k u=v^q,\\
      (-\Delta)^k v=u^p,
   \end{array}
   \right.
$$
has a pair of finite energy positive solutions $(u,v)$,
then
$$
\frac{1}{p+1}+\frac{1}{q+1}=\frac{n-2k}{n}.
$$
\end{corollary}

\begin{proof}
Since $pq>1$, \cite{LGZ}
proved that the solutions $u,v$ of this system satisfy $(-\Delta)^i u
\geq 0$, $(-\Delta)^i v \geq 0$ for $i=1,2,\cdots,k-1$. Thus, this
system is equivalent to the integral system (\ref{hls}) with $\alpha=2k$
(cf. \cite{CL-2010}). According to Theorem \ref{th5.3}, we can also
derive the conclusion.
\end{proof}

\section{Infinite energy solutions}

\subsection{Existence in supercritical case}
For semilinear Lane-Emden equation
(\ref{1.1}), Li \cite{YiLi} obtained a positive solution with
the slow decay rate
$$
u(x)=O(|x|^{-\frac{2}{p-1}}), \quad when ~|x| \to \infty.
$$
According to Corollary \ref{coro1.3}, it is not the finite energy solution.

In this section, we prove that there also exists an infinite energy solution
for bi-Laplace equation in the supercritical case $p>\frac{p+4}{p-4}$.

Clearly,
$$
(-\Delta)^2 u=u^p, \quad in ~R^n,
$$
is equivalent to
$$
 \left \{
   \begin{array}{l}
      -\Delta u=v,\\
      -\Delta v=u^p.
   \end{array}
   \right.
$$
We search the positive solutions with radial structures.
The existence can be implied by the following argument.

\begin{theorem} \label{th6.1}
The following ODE system
\begin{equation}
 \left \{
   \begin{array}{l}
      -(u''+\frac{n-1}{r}u')=v, \quad -(v''+\frac{n-1}{r}v')=u^p,
      \quad r>0\\
      u'(0)=v'(0)=0, \quad u(0)=1, \quad v(0)=a,
   \end{array}
   \right. \label{ODE}          
 \end{equation}
has entire solutions satisfying $\lim_{|x| \to \infty}u(x)
=\lim_{|x| \to \infty}v(x)=0$.
\end{theorem}

\begin{proof}
Here we use the shooting method.

We denote the solutions of (\ref{ODE}) by $u_a(r),v_a(r)$.

{\it Step 1.} By the standard contraction argument, we can see the
local existence.

{\it Step 2.} We claim that for $a \geq 4n$, there exists $R \in (0,1]$ such that
$u_a(r),v_a(r)>0$ for $r \in [0,R)$ and $u_a(R)=0$.

In fact, from (\ref{ODE}) we obtain $u_a'<0$ which implies $u_a(r) \leq u_a(0)=1$, and
$$
v_a(r)=v_a(0)-\int_0^r\tau^{1-n}\int_0^\tau s^{n-1}u_a^p(s)dsd\tau
\geq a-\frac{r^2}{2n} \geq \frac{a}{2}
$$
for $r \in [0,1]$. Therefore,
$$
u_a(r)=u_a(0)-\int_0^r\tau^{1-n}\int_0^\tau s^{n-1}v_a(s)dsd\tau
\leq 1-\frac{ar^2}{4n}.
$$
This proves that for $a \geq 4n$, we can find $R \in (0,1]$ such
that $u_a(r),v_a(r)>0$ for $r \in (0,R)$ and $u_a(R)=0$.

{\it Step 3.} We claim that for $0<a<\varepsilon_0=\frac{1}{n2^{1+p}}$, there
exists $R \in (0,1]$, such that $u_a(r),v_a(r)>0$ for $r \in [0,R)$ and
$v_a(R)=0$.

In fact,
$$
u_a(r) \geq 1-\frac{\varepsilon_0r^2}{2n} \geq \frac{1}{2},
$$
for $r \in (0,1)$. Therefore,
$$
v_a(r) < \varepsilon_0-\frac{1}{2^p}\frac{r^2}{2n}.
$$
This proves that for $a < \varepsilon_0$, we can find $R \in
(0,1]$ such that $u_a(r),v_a(r)>0$ for $r \in (0,R)$ and
$v_a(R)=0$.

{\it Step 4.} Let $\underline{a}=\sup \underline{S}$, where
$$
\underline{S}:=\{\varepsilon; \exists R_a>0, ~\hbox{such that}~
u_a(r)>0, v_a(r) \geq 0, ~\hbox{for}~ r \in [0,R_a], v_a(R_a)=0\}.
$$
Clearly, $\underline{S} \neq \emptyset$ by virtue of $\varepsilon_0
\in \underline{S}$. Noting $\varepsilon \leq a_0$ for $\varepsilon
\in \underline{S}$, we see the existence of $\underline{a}$.

{\it Step 5.} Write $\bar{u}(r)=u_{\underline{a}}(r)$ and
$\bar{v}(r)=v_{\underline{a}}(r)$. We claim that $\bar{u}(r),\bar{v}(r)>0$
for $r \in [0,\infty)$, and hence they are entire positive solutions of
(\ref{ODE}).

Otherwise, there exists $\bar{R}>0$ such that $\bar{u}(r),\bar{v}(r)>0$
for $r \in (0,\bar{R})$ and one of the following consequences holds:

(i) $\bar{u}(\bar{R})=0$, $\bar{v}(\bar{R})>0$;

(ii) $\bar{v}(\bar{R})=0$, $\bar{u}(\bar{R})>0$;

(iii) $\bar{u}(\bar{R})=0$, $\bar{v}(\bar{R})=0$.

We deduce the contradictions from three consequences above.

(i) By $C^1$-continuous dependence of $u_a,v_a$ in $a$, and the fact
$\bar{u}'(\bar{R})<0$, we see that for all $|a-\underline{a}|$ small,
there exists $R_a>0$ such that
$$\begin{array}{ll}
&\bar{u}(r),\bar{v}(r)>0, \quad for ~r \in (0,R_a);\\
&\bar{u}(R_a)=0, \quad \bar{v}(R_a)>0.
\end{array}
$$
This contradicts with the definition of $\underline{a}$.

(ii) Similarly, for $|a-\underline{a}|$ small, there exists $R_a>0$ such that
$$\begin{array}{ll}
&\bar{u}(r),\bar{v}(r)>0, \quad for ~r \in (0,R_a);\\
&\bar{u}(R_a)>0, \quad \bar{v}(R_a)=0.
\end{array}
$$
This implies that $\underline{a}+\delta \in \underline{S}$ for some $\delta>0$,
which contradicts with the definition of $\underline{a}$.

(iii) The consequence implies that $u(x)=\bar{u}(|x|)$ and $v(x)=\bar{v}(|x|)$
are solutions of the system
\begin{equation} \label{bs}
 \left \{
   \begin{array}{l}
      -\Delta u=v, \quad -\Delta v=u^p, ~in ~B_R,\\
      u,v>0 ~in ~B_R, \quad u=v=0 ~on ~\partial B_R.
   \end{array}
   \right.
\end{equation}
It is impossible
by the Pohozaev identity proved later (cf. Theorem \ref{th6.3}).

All the contradictions show that our claim is true. Thus, the entire positive solutions
exist.

{\it Step 6.} We claim $\lim_{r \to \infty}\bar{u}(r),\bar{v}(r)=0$.

Eq. (\ref{ODE}) implies $\bar{u}'<0$ and $\bar{v}'<0$ for $r>0$. So $\bar{u}$
and $\bar{v}$ are decreasing positive solutions, and
$\lim_{r \to \infty}\bar{u}(r)$, $\lim_{r \to \infty}\bar{v}(r)$ exist.

If there exists $c>0$ such that $\bar{v}(r) \geq c$ for $r>0$, then (\ref{ODE})
shows that $\bar{u}$ satisfies
$$
u''+\frac{n-1}{r}u' \leq -c.
$$
Integrating twice yields
$$
\bar{u}(r) \leq \bar{u}(0)-\frac{cr^2}{2n}
$$
for $r>0$. It is impossible since $\bar{u}$ is a entire positive solution.
This shows that $\bar{v} \to 0$ when $r \to \infty$.

Similarly, $u$ has the same property.
\end{proof}

\paragraph{Remark 6.1.}
\begin{enumerate}
\item When $k \in (2,n/2)$ is an integer, the existence of the
$2k$-order PDEs in the supercritical cases is rather challenged.
Recently, Li \cite{Li2011} applied the shooting method and the
analysis of the target map via the degree theory to obtain the
existence results for both (\ref{1.1k}) and (\ref{1.1s}) in
Theorem \ref{th1.1} in the supercritical cases.

\item In the critical case $p=\frac{n+\alpha}{n-\alpha}$,
(\ref{poi}) with $\alpha=2k$ is a solution of (\ref{1.1k}). For
the system (\ref{1.1s}), the critical condition
$\frac{1}{p+1}+\frac{1}{q+1}=1-\frac{\alpha}{n}$ leads to $pq>1$.
The argument in Corollary \ref{coro5.4} shows  the equivalence
between (\ref{1.1s}) and the HLS type system (\ref{hls}).
Therefore, the existence of (\ref{1.1s}) is implied by the
sufficiency of Theorem \ref{th5.3}.

\item In the subcritical case $p<\frac{n+\alpha}{n-\alpha}$, the
nonexistence of positive solutions of (\ref{1.3}) had been proved
(cf. \cite{CGS}, \cite{CLO4} and \cite{Yu}). On the other hand, by
the equivalence between (\ref{1.1k}) and (\ref{1.3}) (cf.
\cite{CL-2010} and \cite{ChLO}), we also see that (\ref{1.1k})
does not exist any positive solution. As regards the nonexistence
for the system (\ref{1.1s}) (or (\ref{1.3s})), it is the
Lane-Emden conjecture (or the HLS conjecture) (cf. \cite{CDM} and \cite{Souplet}).
\end{enumerate}

\subsection{Nonexistence in bounded domain}

In this subsection, we give the Pohozaev identity which the
proof of Theorem \ref{th6.1} needs. In
fact, we can give more general ones which imply nonexistence of
positive solutions of the following $2k$-order PDE (\ref{1.1k})
$$
(-\Delta)^k u=u^p, \quad k\geq 1,~ u>0
$$
with the supercritical exponent $p>\frac{n+2k}{n-2k}$ in any bounded
domain. The argument can help to prove the existence results in
$R^n$.

{\it Note.} Seeing here, we recall another related fact: in the
subcritical case, (\ref{1.1k}) has positive solutions in a bounded
domain. In general, the variational methods works now. However, it
has no positive solution in $R^n$ (cf. Remark 6.1(3)).

\begin{proposition} \label{prop5.2}
Let $D \subset R^n$ be a bounded domain. Assume that
$u_j$ $(j=1,2,\cdots,k)$ solve the following
boundary value problem
\begin{equation}
 \left \{
   \begin{array}{l}
      -\Delta u_1=u_2,\quad -\Delta u_2=u_3,\quad
\cdots,\\
-\Delta u_{k-1}=u_k,\quad -\Delta u_k=u_{k+1}:=u_1^p, \quad in ~D,\\
u_1=u_2=\cdots=u_k=0, \quad on ~\partial D.
   \end{array}
   \right.  \label{bing}         
 \end{equation}
Then
\begin{equation} \label{wu}
\begin{array}{ll}
&\displaystyle\int_Du_1^{p+1}dx=\int_Du_ju_{k+2-j}dx,\quad for~ j=1,2,\cdots,k;\\[3mm]
&\displaystyle\int_Du_1^{p+1}dx=\int_D \nabla u_j \nabla u_{k+1-j}dx,\quad for~ j=1,2,\cdots,k.
\end{array}
\end{equation}
\end{proposition}

\begin{proof}
Applying the boundary value condition, from (\ref{bing}) we obtain
$$\begin{array}{ll}
&\displaystyle\int_D u_1^{p+1}dx=-\int_Du_1\Delta u_kdx=\int_D \nabla u_1 \nabla u_kdx\\[3mm]
&=-\displaystyle\int_D u_k\Delta u_1dx=\int_D u_2u_kdx=-\int_Du_2\Delta u_{k-1}dx\\[3mm]
&=\displaystyle\int_D \nabla u_2 \nabla u_{k-1}dx=\int_D u_3u_{k-1}dx=\cdots\\[3mm]
&=\displaystyle\int_D\nabla u_j\nabla u_{k+1-j}dx =\int_D
u_ju_{k+2-j}dx.
\end{array}
$$
This result implies (\ref{wu}).
\end{proof}

\begin{theorem} \label{th6.3}
Let $D \subset R^n$ be a bounded star-shaped domain. If
\begin{equation} \label{nc}
p \geq \frac{n+2k}{n-2k},
\end{equation}
then the following Navier boundary value problem has no
positive radial solution in $C^{2k}(D)\cap C^{2k-1}(\bar{D})$
\begin{equation}
 \left \{
   \begin{array}{l}
      (-\Delta)^k u=u^p \quad in \quad D,\\
      u=\Delta u=\cdots=\Delta^{k-1}u=0 \quad on \quad \partial D.
   \end{array}
   \right.   \label{hp}       
 \end{equation}

\end{theorem}

\begin{proof}
Clearly, $u=u_1$ satisfies
$$
 \left \{
   \begin{array}{l}
      -\Delta u_1=u_2,~
-\Delta u_2=u_3,~
\cdots,\\
-\Delta u_{k-1}=u_k,~
-\Delta u_k=u_{k+1}:=u_1^p,\quad in ~D,\\
      u_1=u_2=\cdots=u_k=0, \quad on ~\partial D.
   \end{array}
   \right.
$$
By the maximum principle, from $-\Delta u_k=u_1^p>0$ and $u_k|_{\partial
D}=0$, we see $u_k>0$ in $D$. By the same way, we also deduce by induction
that
\begin{equation} \label{positive}
u_j>0 \quad in ~D, \quad j=1,2,\cdots,k.
\end{equation}

Multiplying the $j$-th equation by $(x \cdot \nabla u_{k+1-j})$,
we have
\begin{equation} \label{jia}
\begin{array}{ll}
&-\displaystyle\int_{\partial D}(x\cdot\nu)\partial_\nu u_j \partial_\nu u_{k+1-j}ds
+\int_D\nabla u_j \nabla u_{k+1-j}dx\\[3mm]
&+\displaystyle\int_D x\cdot \nabla u_{jx_i}(u_{k+1-j})_{x_i}dx
=\int_D u_{j+1}(x \cdot \nabla u_{k+1-j})dx
\end{array}
\end{equation}
for $j=1,2,\cdots,k$, where $\nu$ is the unit outward normal vector on $\partial D$.

Integrating by parts, we can see that
$$\begin{array}{ll}
&\displaystyle\int_D x \cdot (u_{jx_i}\nabla (u_{k+1-j})_{x_i}+(u_{k+1-j})_{x_i}
\nabla u_{jx_i})dx\\[3mm]
&=\displaystyle\int_D x \cdot \nabla(\nabla u_j \nabla u_{k+1-j})dx\\[3mm]
&=\displaystyle\int_{\partial D}(x\cdot \nu)\partial_\nu u_j \partial_\nu u_{k+1-j} ds
-n\int_D \nabla u_j \nabla u_{k+1-j} dx,
\end{array}
$$
Combining the results of (\ref{jia}) with $j$ and $k+1-j$, and using the result above,
we deduce that, for $j=1,2,3,\cdots,k$,
\begin{equation} \label{yi}
\begin{array}{ll}
&-\displaystyle\int_{\partial D}(x\cdot \nu)\partial_\nu u_j \partial_\nu u_{k+1-j} ds
+(2-n)\int_D\nabla u_j \nabla u_{k+1-j} dx\\[3mm]
&=\displaystyle\int_D u_{k+2-j}(x\cdot \nabla u_j) dx
+\int_D u_{j+1}(x \cdot \nabla u_{k+1-j})dx.
\end{array}
\end{equation}

Integrating by parts, we also see that for $j=2,3,\cdots,k$,
$$\begin{array}{ll}
&\displaystyle\int_D x\cdot(u_j \nabla u_{k+2-j}+u_{k+2-j} \nabla u_j)dx\\[3mm]
&=\displaystyle\int_D x\cdot \nabla (u_j u_{k+2-j})dx=-n\int_D u_j u_{k+2-j}dx.
\end{array}
$$
Summing $j$ from $1$ to $k$ in (\ref{yi}) and using the result above, we obtain
$$\begin{array}{ll}
&\displaystyle\frac{2-n}{2}\int_D(\nabla u_1\nabla u_k
+\nabla u_2\nabla u_{k-1}+\cdots+\nabla u_k\nabla u_1)dx\\[3mm]
&+\displaystyle\frac{n}{2}\int_D(u_2u_k+u_3u_{k-1}+\cdots+u_ku_2)dx
-\int_D u_{k+1}(x\cdot\nabla u_1)dx\\[3mm]
&=\displaystyle\int_{\partial D}(x\cdot\nu)[(\partial_\nu u_1 \partial_\nu u_k
+\partial_\nu u_2 \partial_\nu u_{k-1}+\cdots+
\partial_\nu u_k \partial_\nu u_1).
\end{array}
$$
By virtue of (\ref{positive}) and the boundary value condition, the Hopf lemma
shows that $\partial_{\nu}u_j<0$ on $\partial D$ for $j=1,2,\cdots,k$.
Noting $D$ is star-shaped, we know that
all terms in the right hand side of the result above are positive. Namely,
\begin{equation} \label{ding}
\begin{array}{ll}
&\displaystyle\frac{2-n}{2}\int_D(\nabla u_1\nabla u_k
+\nabla u_2\nabla u_{k-1}+\cdots+\nabla u_k\nabla u_1)dx\\[3mm]
&+\displaystyle\frac{n}{2}\int_D(u_2u_k+u_3u_{k-1}+\cdots+u_ku_2)dx
+\frac{n}{p+1}\int_D u^{p+1}dx >0.
\end{array}
\end{equation}
Inserting (\ref{wu}) into (\ref{ding}), we have
$$
\frac{n}{p+1}+\frac{k(2-n)}{2}+\frac{n(k-1)}{2}> 0.
$$
This contradicts (\ref{nc}).
\end{proof}

The following result is necessary to prove Theorem \ref{1.1} (2)
(cf. \cite{Li2011}).

\begin{theorem} \label{th6.4}
Let $D \subset R^n$ be a bounded star-shaped domain. If
\begin{equation} \label{nsc}
\frac{1}{p+1}+\frac{1}{q+1} \leq \frac{n-2k}{n},
\end{equation}
then the following Navier boundary value problem has no
positive radial solution in $C^{2k}(D)\cap C^{2k-1}(\bar{D})$
\begin{equation} \label{2ks-b}
 \left \{
   \begin{array}{l}
      (-\Delta)^k u=v^q, \quad (-\Delta)^k v=u^p \quad in ~ D,\\
      u=\Delta u=\cdots=\Delta^{k-1}u=0 \quad on \quad \partial D,\\
      v=\Delta v=\cdots=\Delta^{k-1}v=0 \quad on \quad \partial D.
   \end{array}
   \right.
\end{equation}
\end{theorem}

\begin{proof}
Clearly, the solutions $u_1(=u)$ and $v_1(=v)$ of (\ref{2ks-b}) satisfy
$$
\left \{
   \begin{array}{l}
      -\Delta u_1=u_2, -\Delta u_2=u_3, \cdots, -\Delta u_k=u_{k+1}:=v^q, \quad in ~ D,\\
      -\Delta v_1=v_2, -\Delta v_2=v_3, \cdots, -\Delta v_k=v_{k+1}:=u^p, \quad in ~ D,\\
      u_1=u_2=\cdots=u_k=0 \quad on \quad \partial D,\\
      v_1=v_2=\cdots=v_k=0 \quad on \quad \partial D.
   \end{array}
   \right.
$$
Multiply $-\Delta u_j=u_{j+1}$ and $-\Delta v_j=v_{j+1}$
by $(x \cdot \nabla v_{k+1-j})$ and $(x \cdot \nabla u_{k+1-j})$,
respectively. Integrating by parts yields
\begin{equation} \label{yi1}
\begin{array}{ll}
&-\displaystyle\int_{\partial D}(x\cdot\nu)\partial_\nu u_j \partial_\nu v_{k+1-j} ds
+\int_D \nabla u_j \nabla v_{k+1-j} dx\\[3mm]
&+\displaystyle\int_D [x \cdot \nabla (v_{k+1-j})_{x_i}] u_{jx_i}dx
=\int_D u_{j+1}(x \cdot \nabla v_{k+1-j})dx
\end{array}
\end{equation}
and
\begin{equation} \label{yi2}
\begin{array}{ll}
&-\displaystyle\int_{\partial D}(x\cdot\nu)\partial_\nu v_j \partial_\nu u_{k+1-j} ds
+\int_D \nabla v_j \nabla u_{k+1-j} dx\\[3mm]
&+\displaystyle\int_D [x \cdot \nabla (u_{k+1-j})_{x_i}] v_{jx_i}dx
=\int_D v_{j+1}(x \cdot \nabla u_{k+1-j})dx.
\end{array}
\end{equation}

Adding the $(k+1-j)$-th (\ref{yi1}) and the $j$-th (\ref{yi2}) together leads to
\begin{equation} \label{geng1}
\begin{array}{ll}
&-\displaystyle\int_{\partial D}(x\cdot\nu)\partial_\nu v_1 \partial_\nu u_k ds
+(2-n)\int_D \nabla v_1 \nabla u_k dx\\[3mm]
&=-\displaystyle\frac{n}{q+1}\int_D v_1^{q+1}dx
+\int_D v_2(x \cdot \nabla u_k)dx,
\end{array}
\end{equation}
\begin{equation} \label{geng2}
\begin{array}{ll}
&-\displaystyle\int_{\partial D}(x\cdot\nu)\partial_\nu u_1 \partial_\nu v_k ds
+(2-n)\int_D \nabla u_1 \nabla v_k dx\\[3mm]
&=-\displaystyle\frac{n}{p+1}\int_D u_1^{q+1}dx
+\int_D u_2(x \cdot \nabla v_k)dx,
\end{array}
\end{equation}
and
\begin{equation} \label{geng3}
\begin{array}{ll}
&-\displaystyle\int_{\partial D}(x\cdot\nu)\partial_\nu v_j \partial_\nu u_{k+1-j} ds
+(2-n)\int_D \nabla v_j \nabla u_{k+1-j} dx\\[3mm]
&=\displaystyle\int_D [v_{j+1}(x \cdot \nabla u_{k+1-j})
+u_{k+2-j}(x \cdot \nabla v_j)]dx.
\end{array}
\end{equation}
Summing $j$ from $1$ to $k$, by (\ref{geng1}), (\ref{geng2}) and (\ref{geng3})
we deduce that
$$\begin{array}{ll}
&n\displaystyle\int_D(\frac{u_1^{p+1}}{p+1}+\frac{v_1^{q+1}}{q+1})dx
+n\int_D \sum_{j=2}^k v_ju_{k+2-j} dx\\[3mm]
&+(2-n)\displaystyle\int_D \sum_{j=1}^k \nabla v_j \nabla u_{k+1-j} dx\\[3mm]
&=\displaystyle\int_{\partial D}(x\cdot\nu)\sum_{j=1}^k
\partial_\nu v_j \partial_\nu u_{k+1-j}ds > 0.
\end{array}
$$

Similar to Proposition \ref{prop5.2}, it also follows
$$
\int_D u_1^{p+1} dx=\int_D v_1^{q+1} dx
=\int_D \nabla v_j \nabla u_{k+1-j}dx
=\int_D v_lu_{k+2-l}
$$
for $1 \leq j \leq k$ and $2 \leq l \leq k$.

Combining two results above, we have
$$
\frac{n}{p+1}+\frac{n}{q+1}+n(k-1)+(2-n)k > 0,
$$
which contradicts (\ref{nsc}).
\end{proof}

\paragraph{Acknowledgements.} The work of Y. Lei is partially supported by NSFC grant 11171158,
the Natural Science Foundation of Jiangsu (BK2012846) and SRF for ROCS, SEM.
The work of C. Li is partially supported by NSF grant DMS-0908097 and NSFC grant 11271166.

Yutian Lei

Institute of Mathematics, School of Mathematical Sciences,
Nanjing Normal University, Nanjing, 210097, China\\

Congming Li

Department of Applied Mathematical, University of Colorado at Boulder,
Boulder, CO 80309, USA

Department of Mathematics, and MOE-LSC, Shanghai Jiao Tong University,
Shanghai, 200240, China

\end{document}